\documentclass[10pt]{amsart}
\usepackage{amsmath,amssymb,amsthm}
\usepackage{enumerate}                
\usepackage{srcltx}                   
\usepackage{cite}                     
\usepackage{xspace}                   
\usepackage[margin=1in]{geometry}     
\usepackage{mathabx}                  
\usepackage{mathtools}                
\usepackage{ctable}                   
\usepackage{multirow}                 

\newtheorem{theorem}{Theorem}[section]

\newtheorem{corollary}[theorem]{Corollary}
\newtheorem{lemma}[theorem]{Lemma}

\theoremstyle{definition}
\newtheorem{definition}[theorem]{Definition}
\theoremstyle{remark}
\newtheorem{remark}[theorem]{Remark}

\newcommand*{\fp}[3]{\Asterisk_{#1 \in #2} #3_{#1}}
\newcommand*{\freeprod}{\fp{a}{A}{\Gamma}}
\newcommand*{\closure}[2][3]{{}\mkern#1mu\overline{\mkern-#1mu#2}}
\newcommand*{\wlen}[1]{||#1||}
\newcommand*{\card}[1]{|#1|}
\newcommand*{\freepr}{\mathbin{*}}

\newlength{\barheight}
\newsavebox{\tempbox}
\newsavebox{\tempvbox}
\savebox{\tempvbox}{\vbox to 7pt{}}

\def\daggerform#1{%
  \def\deta##1{##1}%
  \def\detai##1{##1}%
  \def\dmu##1{##1}%
  \def\dmui##1{##1}%
  \def\dnu##1{##1}%
  \def\dnui##1{##1}%
  \savebox{\tempbox}{\vphantom{$\displaystyle#1$}}%
  \def\dmu##1{\underbracket[0.3pt][\barheight]{##1\usebox{\tempbox}}_{\mu\usebox{\tempvbox}}}%
  \def\dmui##1{\underbracket[0.3pt][\barheight]{##1\usebox{\tempbox}}_{\mu^{-1}\usebox{\tempvbox}}}%
  \def\dnu##1{\underbracket[0.3pt][\barheight]{##1\usebox{\tempbox}}_{\nu\usebox{\tempvbox}}}%
  \def\dnui##1{\underbracket[0.3pt][\barheight]{##1\usebox{\tempbox}}_{\nu^{-1}\usebox{\tempvbox}}}%
  \def\deta##1{\underbracket[0.3pt][\barheight]{##1\usebox{\tempbox}}_{\eta\usebox{\tempvbox}}}%
  \def\detai##1{\underbracket[0.3pt][\barheight]{##1\usebox{\tempbox}}_{\eta^{-1}\usebox{\tempvbox}}}%
#1}

\def\symform#1{%
  \def\lpart##1{##1}%
  \def\rpart##1{##1}%
  \savebox{\tempbox}{\vphantom{$\displaystyle#1$}}%
  \def\lpart##1{\underbracket[0.3pt][\barheight]{##1\usebox{\tempbox}}}%
  \def\rpart##1{\underbracket[0.3pt][\barheight]{##1\usebox{\tempbox}}}%
#1}

\begin{document}
\title{Automatic continuity for homomorphisms into free products}

\author{Konstantin Slutsky}
\address{Department of Mathematical Sciences\\
University of Copenhagen\\
Universitetsparken 5\\
2100 K\o benhavn \O\\
Denmark}
\email{kslutsky@gmail.com}
\thanks{Research supported by Denmark's Council for Independent Research (Natural Sciences Division), grant
  no. 10-082689/FNU}
\begin{abstract}
  A homomorphism from a completely metrizable topological group into a free product of groups whose image is not
  contained in a factor of the free product is shown to be continuous with respect to the discrete topology on the
  range.  In particular, any completely metrizable group topology on a free product is discrete.
\end{abstract}

\maketitle

\section{Introduction and background}
\label{sec:introduction}

The main concept of automatic continuity is to establish conditions on topological groups \( G \) and \( H \) under
which a homomorphism \( \phi : G \to H \) is necessarily continuous.  In many cases one of the groups is assumed to be
very special, while the conditions on the other group are relatively mild.  An important example of this phenomenon is
a theorem of R.~M.~Dudley\cite{MR0136676} from 1961:
\begin{theorem}[Dudley]
  \label{thm:Dudley-theorem}
  Any homomorphism from a completely metrizable group or a Hausdorff locally compact group into a free group with the
  discrete topology is continuous. 
\end{theorem}
In fact Dudley's result is more general: the target group does not necessarily have to be a free group, it is enough to
assume that it admits a special length function (which is called a norm in \cite{MR0136676}).  In the case of a free
group this length function is given by the length of the reduced form of an element.  An important corollary of this
theorem is that any Hausdorff locally compact or completely metrizable group topology on a free group is necessarily
discrete.  In 1976 S.~A.~Morris and P.~Nickolas \cite{MR0399341} investigated continuity of homomorphisms
\( \phi : G \to \Gamma \) between a Hausdorff locally compact \( G \) and free products \( \Gamma = A \freepr B \).  
\begin{theorem}[Morris--Nickolas]
  \label{thm:morris-nickolas}
  Given a homomorphism \( \phi : G \to A \freepr B \) from a Hausdorff locally compact group into a free
  product with the discrete topology either \( \phi \) is continuous or its image \( \phi(G) \) is contained (up to a
  conjugation) in one of the factors.  In particular, any Hausdorff locally compact group topology on a non-trivial free
  product in discrete.
\end{theorem}
The next step was made six years later by R.~Alperin in \cite{MR661438}, where he studied actions of Hausdorff locally
compact groups on trees in the sense of Bass--Serre theory.  Automatic continuity implications of Alperin's work are
summarized in the following theorem.

\begin{theorem}[Alperin]
  \label{thm:alperin}
  Let \( G \) be a locally compact group and \( G_{0} \) be the connected component of the identity.
  \begin{enumerate}[(i)]
  \item If \( G \) can be written in a non-trivial way as an amalgam \( A \freepr_{C} B \), then \( A
    \), \( B \) and \( C \) are necessarily open.
  \item\label{item:locally-compact-alpering-free-prod} If \( G/G_{0} \) is compact and \( \phi : G \to A \freepr B \) is
    a homomorphisms into a free product, then up to a conjugation \( \phi(G) \) is contained in one of the
    factors.
  \end{enumerate}
\end{theorem}
Methods used by Morris--Nickolas and Alperin rely on the structure theory for locally compact groups, and
therefore are not applicable for general completely metrizable groups.  C.~Rosendal \cite{MR2849256} developed further
the Bass--Serre theory approach and proved the following. 
\begin{theorem}[Rosendal]
  \label{thm:rosendal-summary}
  Let \( G \) be a completely metrizable group, and let \( \mathbb{D} \) be the set of elements that generate a finite
  or non-discrete subgroup:
  \[  \mathbb{D} = \left\{ g \in G : \langle g \rangle \ \textrm{is finite or non-discrete} \right\}.  \]
    \begin{enumerate}[(i)]
    \item If \( \mathbb{D} \) is somewhere dense and if \( G \) is written in a non-trivial way as an amalgam
      \( A \freepr_{C} B \), then the three subgroups \( A \), \( B \) and \( C \) are open in \( G \).
    \item If \( \mathbb{D} \) is dense in \( G \) and if \( \phi : G \to A \freepr B \) is a homomorphism into a
      free product, then up to a conjugation \( \phi(G) \) is contained in one of the factors.
    \end{enumerate}
\end{theorem}

The aforementioned articles can be put into a table: 

\ctable[
botcap,
caption = {Automatic continuity results for homomorphisms \( \phi : G \to \Gamma \).},
captionskip = 0.5ex,
label = table:results-summary,
pos = ht,
]{rlcc}{}
{ \FL
& \multirow{2}{*}{The group \( \Gamma \) is} & The group \( G \) is & The group \( G \) is  \NN
&                                        & locally compact      & completely metrizable \ML 
\addlinespace[1.5ex]
& a free group                           & Dudley 1961           & Dudley 1961 \NN \addlinespace[0.5ex]
& a free product \( A \freepr B \) \quad & Morris--Nickolas 1976 & --- \NN \addlinespace[0.5ex]
& an amalgam \( A \freepr_{C} B \)        & Alperin 1982          & Rosendal 2011 \LL \addlinespace[0.5ex]
}

Our goal in this paper is to fill the gap in the table and to prove the main result of Morris--Nickolas \cite{MR0399341}
for completely metrizable groups.  

We would like to mention that none of the results in the table fully imply any other result in it.  For example, let us
take a locally compact group \( G \) and a homomorphism \( \phi : G \to F \) into a free group \( F \).  If we view the
free group \( F \) as a free product of copies of \( \mathbb{Z} \) and apply Morris--Nickolas' Theorem (which is, in
fact, true for any number of factors in the free product), then we may conclude that either \( \phi \) is continuous or
\( \phi(G) \) is contained in a cyclic subgroup of \( F \).  But to get Dudley's Theorem we still need to know that all
homomorphisms \( \phi: G \to \mathbb{Z} \) are continuous.

\bigskip

We now state our main results.  Let \(\Gamma = \freeprod \) and let \( \phi : G \to \Gamma\) be a homomorphism
from a completely metrizable group into a free product of non-trivial groups.  We say that \emph{the image of \( \phi \)
  is contained in a factor of \( \Gamma \)} if there are \( a \in A \) and \( \gamma \in \Gamma \) such that
\( \phi(G) \subseteq \gamma \Gamma_{a} \gamma^{-1} \).  We shall prove (see Theorem
\ref{thm:homomorphisms-into-free-products-are-continuous-or-contain-in-a-factor} below)

\begin{theorem}
  \label{thm:continuity-into-free-products}
  Any homomorphism \( \phi : G \to \freeprod \) from a completely metrizable group into a free product with the discrete
  topology is continuous unless its image is contained in one of the factors.
\end{theorem} 

By taking \( \phi \) to be the identity map we obtain the following corollary (which, in fact, will be proved before
Theorem \ref{thm:continuity-into-free-products} and will serve as a step in its proof).

\begin{theorem}
  \label{thm:completely-metrizable-free-product-discrete}
  Any completely metrizable group topology on a non-trivial free product \( \freeprod \) is discrete.
\end{theorem}

Our proof of Theorems \ref{thm:continuity-into-free-products} and \ref{thm:completely-metrizable-free-product-discrete}
will take the original construction of Dudley as an important ingredient, so we find it useful to recall briefly the
argument from \cite{MR0136676} for completely metrizable groups.
\begin{proof}[Sketch of proof of Dudley's Theorem]
  Let \( \phi: G \to F \) be a homomorphism from a completely metrizable group \( G \) into a free group \( F \) and let
  \( ||\cdot || : F \to \mathbb{N} \) be the length function given by the length of the reduced form of an element.  The
  crucial property of this length function is the inequality \( ||\gamma^{n}|| \ge \max\big\{n, ||\gamma||\big\} \) for
  all non-trivial \( \gamma \in F \) and all \( n \ge 1 \).

  If the homomorphism \( \phi \) is discontinuous, then its kernel cannot be open, and we can find a sequence
  \( (g_{m})_{m=1}^{\infty} \) of elements in \( G \) such that \( \phi(g_{m}) \ne e \)  and for
  \[ \widetilde{h}_{n,m} = g_{m}(g_{m+1}( \cdots (g_{n-1}(g_{n})^{r_{n-1}})^{r_{n-2}} \cdots)^{r_{m+1}})^{r_{m}}, \]
  where \( r_{n} = n + \sum_{l=1}^{n} ||\phi(g_{l})|| \), the sequence \( \widetilde{h}_{n,m} \) converges for all \( m \).
  For the limit \( h_{m} = \lim_{n \to \infty} \widetilde{h}_{n,m} \) we have \( h_{m} = g_{m}h_{m+1}^{r_{m}} \).  We claim that
  \( ||\phi(h_{m})^{r_{m-1}}|| \ge r_{m-1} \).  Indeed, if \( \phi(h_{m+1}) \ne e \), then
  \[ ||\phi(h_{m})^{r_{m-1}}|| \ge ||\phi(h_{m})|| \ge ||\phi(h_{m+1})^{r_{m}}||-||\phi(g_{m})|| \ge r_{m-1} + 1. \]
  If on the other hand \( \phi(h_{m+1}) = e \), then \( \phi(h_{m})^{r_{m-1}} = \phi(g_{m})^{r_{m-1}} \) and
  \( ||\phi(g_{m})^{r_{m-1}}|| \ge r_{m-1} \), since \( \phi(g_{m}) \ne e \) by the choice of \( g_{m} \).  And so
  \begin{displaymath}
    \begin{aligned}
      ||\phi(h_{1})|| &= ||\phi(g_{1}) \phi(h_{2})^{r_{1}}|| \ge ||\phi(h_{2})^{r_{1}}|| - ||\phi(g_{1})|| \ge
      ||\phi(h_{2})|| - ||\phi(g_{1})||\\
      &= ||\phi(g_{2})\phi(h_{3})^{r_{2}}|| - ||\phi(g_{1})|| \ge ||\phi(h_{3})^{r_{2}}|| - ||\phi(g_{1})|| -
      ||\phi(g_{2})|| \ge ||\phi(h_{3})|| - ||\phi(g_{1})|| - ||\phi(g_{2})||\\[9pt]
      &=\quad \cdots \\
      &= ||\phi(g_{m})\phi(h_{m+1})^{r_{m}}|| - \sum_{l=1}^{m-1} ||\phi(g_{l})|| \ge ||\phi(h_{m+1})^{r_{m}}|| -
      \sum_{l=1}^{m}||\phi(g_{l})|| \ge r_{m} - \sum_{l=1}^{m}||\phi(g_{l})|| = m
    \end{aligned}
  \end{displaymath}
  is true for all \( m \), which is impossible.
\end{proof}

In the above argument we used the freeness of the target group only in the amount that it implies the existence of a
length function with the special property.  In general, a \emph{length function} on a group \( \Gamma \) is a function
\( ||\cdot|| : \Gamma \to \mathbb{N} \) that for all \( \gamma, \gamma_{1}, \gamma_{2} \in \Gamma \) satisfies:
\begin{enumerate}[(i)]
  \item \( ||e|| = 0 \) (where \( e \) is the identity element);
  \item \( ||\gamma^{-1}|| = ||\gamma|| \);
  \item \( ||\gamma_{1} \gamma_{2} || \le ||\gamma_{1}|| + ||\gamma_{2}||\).
\end{enumerate}
Suppose now that we have a homomorphism \( \phi : G \to \Gamma \) and a sequence
\( (g_{m})_{m=1}^{\infty} \) of elements in \( G \) such that the limit
\[ h_{m} = \lim_{n \to \infty} g_{m}(g_{m+1}(\cdots (g_{n-1}(g_{n})^{r_{n-1}})^{r_{n-2}} \cdots)^{r_{m+1}})^{r_{m}} \]
exists for every \( m \), where \( r_{n} \) is an increasing sequence of natural numbers (possibly different from the
sequence used in the argument above).  To obtain the inequalities \( ||h_{1}|| \ge M \) for all \( M \) with respect to
a given length function \( ||\cdot|| \) on \( \Gamma \) we need the following:
\begin{enumerate}[(i)]
\item \( ||\phi(h_{m})^{r_{m-1}}|| \ge r_{m-1} \) for all \( m \ge 2 \);
\item \( ||\phi(h_{m})^{r_{m-1}}|| \ge ||\phi(h_{m})|| \) for all \( m \ge 2 \);
\item numbers \( r_{m} - \sum_{l=1}^{m}||\phi(g_{l})|| \) are unbounded from above.
\end{enumerate}

In our setting the group \( \Gamma = \freeprod \) will be a free product of non-trivial groups and the length function
\( ||\cdot|| : \freeprod \to \mathbb{N} \) will be given by the length of the reduced form of an element.  For such a
length function the inequality \( ||\gamma^{n}|| \ge \max\big\{n, ||\gamma||\big\} \) is false for some non-trivial
\( \gamma \in \Gamma \), so we shall need to take special care to fulfil the above list of conditions.

In Section \ref{sec:length-function} we study the length function on \( \freeprod \) and the set of elements on which
the inequality \( ||\gamma^{n}|| \ge \max\big\{n, ||\gamma||\big\} \) fails.  In Section \ref{sec:free-prod-discrete} we
show that any completely metrizable group topology on a free product is discrete.  Section \ref{sec:more-calc-free} does
some further analysis of the behaviour of the length function and Section \ref{sec:autom-cont-homom} contains the proof
of the main result.

\medskip

The author wants to thank Christian Rosendal for helpful discussions of the automatic continuity and for noticing an
error in an earlier draft.  The author also thanks Steven Deprez for his \TeX\ advice and the anonymous referee for
a careful reading of the paper.

\section{Word length functions on free products}
\label{sec:length-function}

Throughout this section let \( \Gamma = \freeprod \) be a free product of non-trivial groups \( \Gamma_{a} \) with
\( A \) being an index set of size \( \card{A} \ge 2 \).  The identity element of any group is denoted by \( e \); we
hope that there will be no confusion to which of the groups we refer.  Let
\( X_{\Gamma} = \bigcup_{a \in A} \Gamma_{a} \) viewed as a subset of \( \Gamma \).  The set \( X_{\Gamma} \) depends
not only on the group \( \Gamma \), but also on the representation of \( \Gamma \) as a free product \( \freeprod \),
but we abuse notations and write just \( X_{\Gamma} \); the choice of \( \Gamma_{a} \) will be clear from the context.
We say that two elements \( x,y \in X_{\Gamma} \) are \emph{multipliable} if there is an \( a \in A \) such that
\( x,y \in \Gamma_{a} \); in this case there is a unique \( z \in X_{\Gamma} \) such that \( x \cdot y = z \).  Any
element \( \gamma \in \Gamma \) can be written uniquely as a reduced product \( \gamma = x_{1} x_{2} \cdots x_{n} \),
\( x_{i} \in X_{\Gamma} \), where reduced means \( x_{i} \) and \( x_{i+1} \) are not multipliable.  The word length
function on \( \Gamma \) is then defined by \( \wlen{\gamma} = n \).  We also make an agreement that \( e \in \Gamma \)
is represented by the empty word and \( \wlen{e} = 0 \).  The product
\( \gamma = \gamma_{1} \gamma_{2} \cdots \gamma_{n} \) for \( \gamma_{i} \in \Gamma \) is said to be reduced if
\( \wlen{\gamma} = \sum_{i=1}^{n}\wlen{\gamma_{i}} \).  An element \( \gamma \in \Gamma \) is called \emph{cyclically
  reduced} if in its reduced form \( \gamma = x_{1} \cdots x_{n} \) either \( n \le 1 \) or \( x_{1} \ne x_{n}^{-1} \).

\begin{lemma}
  \label{lem:core-of-a-word}
  Any \( \gamma \in \Gamma \) can be uniquely written as a reduced product of the form
  \( \gamma = \alpha \gamma_{c} \alpha ^{-1} \), \( \alpha, \gamma_{c} \in \Gamma \), where \( \gamma_{c} \) is
  cyclically reduced.
\end{lemma}
\begin{proof}
  Let \( \gamma \in \Gamma \) be a non-trivial element with the reduced form \( y_{1} \ldots y_{N} \),
  \( y_{i} \in X_{\Gamma} \).  Let \( k \le N/2 \) be maximal such that \( y_{k} = y_{N-k+1}^{-1} \) and set
  \( \alpha = y_{1} \cdots y_{k} \), \( \gamma_{c} = y_{k+1} \cdots y_{N-k} \).  Then
  \( \gamma = \alpha \gamma_{c} \alpha^{-1} \) is reduced and \( y_{k+1} y_{N-k} \ne e \) by the choice of \( k \)
  unless \( N-k = k+1 \), in which case \( k = (N-1)/2 \) and \( \wlen{\gamma_{c}} = 1 \).
\end{proof}

We introduce the following sets to control the behaviour of the length function on \( \Gamma \):
\begin{displaymath}
  \begin{aligned}
    S_{\Gamma} &= \{ \gamma \in \Gamma : \wlen{\gamma^{n}} \ge n\ \textrm{for all \( n \ge 1 \)}\},  \\
    \widetilde{F}_{\Gamma} &= \{ \gamma \in \Gamma : \wlen{\gamma^{n}} \ge \wlen{\gamma}\ \textrm{for all \( n \ge 1
      \)}\}, \\
    F_{\Gamma} &= \{ \gamma \in \Gamma: \gamma^{n} = e\ \textrm{for some \( n \ge 1 \)} \}.
  \end{aligned}
\end{displaymath}
Complements of \( F_{\Gamma} \) and \( S_{\Gamma} \) in \( \Gamma \) will be denoted by \( F^{c}_{\Gamma} \) and \(
S^{c}_{\Gamma} \) respectively.
\begin{lemma}
  \label{lem:description-of-S-and-F}
  For any group \( \Gamma = \freeprod\)
  \begin{enumerate}[(i)]
  \item\label{item:description-S}
    \( S^{c}_{\Gamma} = \{ \alpha x \alpha^{-1} : \alpha \in \Gamma,\, x \in X_{\Gamma} \} \), and in particular
    \( \gamma \in S^{c}_{\Gamma} \) if and only if \( \gamma^{n} \in S^{c}_{\Gamma} \) for all \( n \).
  \item\label{item:when-power-is-not-shorter} For any \( n \ge 2 \) and any \( \gamma \in \Gamma \) the inequality
    \( \wlen{\gamma^{n}} < \wlen{\gamma} \) holds if and only if \( \gamma = \alpha x \alpha^{-1} \),
    \( x \in X_{\Gamma} \), with \( x^{n} = e \) and \( x \ne e \).  In other words
    \( \wlen{\gamma^{n}} \ge \wlen{\gamma} \) unless \( \gamma^{n} = e \).
  \item\label{item:description-F}
    \( \widetilde{F}_{\Gamma}^{c} = \{ \alpha x \alpha^{-1} : \alpha \in \Gamma,\, x \in X_{\Gamma},\, x^{n} = e\
    \textrm{for some \( n \ge 1 \)},\, x \ne e \} = F_{\Gamma}\setminus\{e\}\).
  \end{enumerate}
\end{lemma}
\begin{proof}
  \eqref{item:description-S} If \( \gamma = \alpha x \alpha^{-1} \) for some \( x \in X_{\Gamma} \), then
  \( \wlen{\gamma^{n}} = \wlen{\alpha x^{n} \alpha^{-1}} \le \wlen{\alpha x \alpha^{-1}} = \wlen{\gamma} \) and
  therefore \( \gamma \not \in S_{\Gamma} \).  If \( \gamma = \alpha \gamma_{c} \alpha^{-1} \) is reduced and
  \( \wlen{\gamma_{c}} \ge 2 \), then
  \[ \wlen{\gamma^{n}} = \wlen{\alpha \gamma_{c}^{n} \alpha^{-1}} \ge 2\wlen{\alpha} + n\wlen{\gamma_{c}} - (n-1) >
  n, \] and \( \gamma \in S_{\Gamma} \).

  \eqref{item:when-power-is-not-shorter} If \( \gamma = \alpha x \alpha^{-1} \) with \( x \in X_{\Gamma} \),
  \( x \ne e \) and \( x^{n} = e \), then \( \gamma^{n} = e \) and in particular \( \wlen{\gamma^{n}} < \wlen{\gamma}
  \).  If \( x^{n} \ne e \), then \( \wlen{\gamma^{n}} = \wlen{\gamma} \).  Finally, if
  \( \gamma = \alpha \gamma_{c} \alpha^{-1} \) and \( \wlen{\gamma_{c}} \ge 2 \), then for any \( n \ge 2 \) we have
  \[ \wlen{\gamma^{n}} \ge 2\wlen{\alpha} + n\wlen{\gamma_{c}} - (n-1) = 2\wlen{\alpha} + n(\wlen{\gamma_{c}}-1) + 1 \ge
  2\wlen{\alpha} + 2\wlen{\gamma_{c}} -1 \ge \wlen{\gamma}. \]

  \eqref{item:description-F} The equality
  \( F_{\Gamma} = \{ \alpha x \alpha^{-1} : \alpha \in \Gamma,\ x \in X_{\Gamma},\ x^{n} = e\ \textrm{for some
    \( n \ge 1 \)} \} \) is classical, so \( \widetilde{F}_{\Gamma}^{c} = F_{\Gamma} \setminus\{e\} \) follows from item
  \eqref{item:when-power-is-not-shorter}.
\end{proof}

We now come to a rather technical lemma that describes elements that can be written as a product of two elements from
\( S^{c}_{\Gamma} \).  Corollaries of this lemma will serve as supplements to our modifications of Dudley's argument.
\begin{lemma}
  \label{lem:description-of-elements-in-S-squared}
  Let \( \gamma \in S_{\Gamma}^{c} \cdot S_{\Gamma}^{c} \) be a non-trivial element.  There are
  \( \mu, \nu, \eta \in \Gamma \) such that the reduced form of \( \gamma \) is one of the following:
  \begin{enumerate}[(i)]
  \item\label{item:trivial} \( \eta z \eta ^{-1} \),\quad for some \( z \in X_{\Gamma}\setminus \{e\} \);
  \item\label{item:simple} \( \eta \mu z_{0} \mu^{-1} \nu z_{1} \nu^{-1} \eta^{-1} \),\quad for some
    \( z_{0}, z_{1} \in X_{\Gamma}\setminus \{e\} \);
  \item\label{item:mixed} \( \eta \delta_{1} \nu z \nu^{-1} \delta_{2} \eta^{-1} \),\quad for some
    \( \delta_{1}, \delta_{2}, z \in X_{\Gamma}\setminus \{e\} \), \( \delta_{1} \) and \( \delta_{2} \) are
    multipliable and \( \delta_{1} \cdot \delta_{2} \ne e\);
  \item\label{item:full}
    \( \eta \delta_{1} \mu z_{0} \mu^{-1} \delta_{2} \nu z_{1} \nu^{-1} \delta_{3} \eta^{-1} \),\quad for
    \( \delta_{i}, z_{j} \in X_{\Gamma}\setminus \{e\} \), \( \delta_{1},\delta_{2},\delta_{3} \) are multipliable and
    \( \delta_{2} = \delta_{1}^{-1} \cdot \delta_{3}^{-1} \).
  \end{enumerate}
\end{lemma}
\noindent Remember that we treat identity as an empty word, so any of the \( \mu, \nu, \eta \) are
allowed to be absent.
\begin{proof}
  Fix \( \gamma \in S_{\Gamma}^{c} \cdot S_{\Gamma}^{c} \), \( \gamma \ne e \),
  \( \gamma = \alpha z_{0} \alpha^{-1} \cdot \beta z_{1} \beta^{-1} \), where \( z_{0}, z_{1} \in X_{\Gamma} \) and both
  \( \alpha z_{0} \alpha^{-1} \) and \( \beta z_{1} \beta^{-1} \) are reduced.  If either \( z_{0} = e \) or
  \( z_{1} = e \), then the reduced form of \( \gamma \) is described by item \eqref{item:trivial} and we shall assume
  that \( z_{0} \) and \( z_{1} \) are non-trivial.  Let \( \alpha = x_{1} \cdots x_{m} \),
  \( \beta = y_{1} \cdots y_{n} \), \( x_{i}, y_{j} \in X_{\Gamma} \) be the reduced forms.  By taking \( \gamma^{-1} \)
  instead of \( \gamma \) we may and shall assume that \( m \le n \).  We are going to show that no matter how many
  reductions take place in the product \( \alpha z_{0} \alpha^{-1} \cdot \beta z_{1} \beta^{-1} \), the element's
  \( \gamma \) reduced representation has form given by one of the items above.  To make it easier to recognise that a
  given product can indeed be written in a certain form we shall underline elements \( \eta \), \( \mu \), \( \nu \) and
  their inverses; elements \( z_{j} \) and \( \delta_{i} \) are then the letters in between.
  
  If \( \gamma = \alpha z_{0} \alpha^{-1} \beta z_{1} \beta^{-1} \) is reduced, then, of course, \( \gamma \) has the
  form given by item \eqref{item:simple}.  So we assume that this product is not reduced.  If \( m = 0 \) and
  \( n = 0 \), then \( \gamma \) is a single letter, which fits in \eqref{item:trivial}.  If \( m = 0 \), but
  \( n \ge 1 \), then, depending on whether \( z_{0} \cdot y_{1} = e \), the reduced form of \( \gamma \) is one of the
  two:
  \begin{displaymath}
    \begin{aligned}
     \gamma &= \daggerform{\dmu{y_{2} \cdots y_{n}} z_{1} \dmui{y_{n}^{-1} \cdots y_{2}^{-1}} z_{0}},\\
      \gamma &= \daggerform{u \dnu{y_{2} \cdots y_{n}} z_{1} \dnui{y_{n}^{-1} \cdots y_{2}^{-1}} y_{1}^{-1}},\quad u =
      z_{0} \cdot y_{1},
    \end{aligned}
  \end{displaymath}
  which satisfy items \eqref{item:simple} and \eqref{item:mixed} respectively.

  So we assume \( m \ge 1 \), \( n \ge 1 \), and \( x_{1} \), \( y_{1} \) are
  multipliable.  If \( x_{1}^{-1} \cdot y_{1} \ne e \), then the product
  \[ \gamma = \daggerform{x_{1} \dmu{x_{2} \cdots x_{m}} z_{0} \dmui{x_{m}^{-1} \cdots x_{2}^{-1}} u \dnu{y_{2} \cdots
      y_{n}} z_{1} \dnui{y_{n}^{-1} \cdots y_{2}^{-1}} y_{1}^{-1}},\quad u = x_{1}^{-1} \cdot y_{1}, \]
  is reduced and \( \gamma \) satisfies \eqref{item:full}.  If \( x_{1} = y_{1} \), let \( k \le m\) be maximal such that
  \( x_{i} = y_{i} \) for \( i \le k \).  If \( k < m\), then, depending on whether \( x_{k+1} \) and
  \( y_{k+1} \) are multipliable or not, one of the two products
  \begin{displaymath}
    \begin{aligned}
      \gamma &= \daggerform{\deta{y_{1} \cdots y_{k}} \dmu{x_{k+1} \cdots x_{m}} z_{0} \dmui{x_{m}^{-1} \cdots
          x_{k+1}^{-1}} \dnu{y_{k+1} \cdots
          y_{n}} z_{1} \dnui{y_{n}^{-1} \cdots y_{k+1}^{-1}} \detai{y_{k}^{-1} \cdots y_{1}^{-1}}}, \\
      \gamma &= \daggerform{\deta{y_{1} \cdots y_{k}} x_{k+1} \dmu{x_{k+2} \cdots x_{m}} z_{0} \dmui{x_{m}^{-1} \cdots
          x_{k+2}^{-1}} u \dnu{y_{k+2} \cdots y_{n}} z_{1} \dnui{y_{n}^{-1} \cdots y_{k+2}^{-1}} y_{k+1}^{-1}
        \detai{y_{k}^{-1} \cdots y_{1}^{-1}}},\quad u = x_{k+1}^{-1} \cdot y_{k+1},
    \end{aligned}
  \end{displaymath}
  is reduced and they are of the form \eqref{item:simple} and \eqref{item:full} respectively. 
  
  We therefore consider the case \( k = m \).  If \( m = n \), then depending on whether \( z_{0} \) and \( z_{1} \) are
  multipliable the reduced form of \( \gamma \) is one of the following two
  \begin{displaymath}
    \begin{aligned}
      \gamma &= \daggerform{ \deta{y_{1} \cdots y_{m}} z_{0} z_{1} \detai{y_{m}^{-1} \cdots y_{1}^{-1}} }, \\
      \gamma &= \daggerform{ \deta{y_{1} \cdots y_{m}} u \detai{y_{m}^{-1} \cdots y_{1}^{-1}} },
      \quad u = z_{0} \cdot z_{1} \ne e, \\
    \end{aligned}
  \end{displaymath}
  and they satisfy \eqref{item:simple} and \eqref{item:trivial} respectively (here we use \( \gamma \ne e \) to
  exclude \( z_{0} \cdot z_{1} = e \)).

  We may now assume that \( m < n \) and \( \gamma \) can be written (in a not necessarily reduced way) as
  \[ \gamma = \daggerform{\deta{y_{1} \cdots y_{m}} z_{0} \cdot \dnu{y_{m+1} \cdots y_{n}} z_{1} \dnui{y_{n}^{-1} \cdots
      y_{m+1}^{-1}} \detai{y_{m}^{-1} \cdots y_{1}^{-1}}}, \]
  where \( \cdot \) represents the only position where reductions can take place.  If \( z_{0} \) is not multipliable
  with \( y_{m+1} \), then the above representation is reduced and satisfies \eqref{item:simple}.  If it is not reduced,
  but \( z_{0} \cdot y_{m+1} \ne e \), then the reduced form of \( \gamma \) is
  \[ \gamma = \daggerform{\deta{y_{1} \cdots y_{m}} u \dnu{y_{m+2} \cdots y_{n}} z_{1} \dnui{y_{n}^{-1} \cdots
      y_{m+2}^{-1}} y_{m+1}^{-1} \detai{y_{m}^{-1} \cdots y_{1}^{-1}}},\quad u = z_{0} \cdot y_{m+1},  \]
  in accordance with \eqref{item:mixed}.
  If \( z_{0} = y_{m+1}^{-1} \), then setting \( s_{0} = n+1 \), \( s_{1} = s_{0} - (n-m) = m+1 \) we have
  \[ \gamma = y_{1} \cdots y_{s_{1}-1} \cdot y_{s_{1}+1} \cdots y_{s_{0}-1} z_{1} y_{s_{0}-1}^{-1} \cdots
  y_{s_{1}+1}^{-1} z_{0} y_{s_{1}-1}^{-1} \cdots y_{1}^{-1}.  \]
  If the block \( y_{s_{1}+1} \cdots y_{s_{0}-1} z_{1} \) cancels the block \( y_{s_{2}} \cdots y_{s_{1}-1} \), where
  \( s_{2} = s_{1} - (n-m) = s_{0} - 2(n-m) \), then
  \[ \gamma = y_{1} \cdots y_{s_{2}-1} \cdot y_{s_{2}+1} \cdots y_{s_{1}-1} z_{0} y_{s_{1}-1}^{-1} \cdots
  y_{s_{2}+1}^{-1} z_{1} y_{s_{2}-1}^{-1} \cdots y_{1}^{-1}.  \]
  If again the block \( y_{s_{2}+1} \cdots y_{s_{1}-1} z_{0} \) cancels the block \( y_{s_{3}} \cdots y_{s_{2}-1} \) we
  can proceed in the same way and build a sequence \( s_{k} = s_{k-1} - (n-m) = s_{0} - k(n-m) \) such that
  \[ \gamma = y_{1} \cdots y_{s_{k}-1} \cdot y_{s_{k}+1} \cdots y_{s_{k-1}-1} z_{\epsilon_{k}} y_{s_{k-1}-1}^{-1} \cdots
  y_{s_{k}+1}^{-1} z_{1-\epsilon_{k}} y_{s_{k}-1}^{-1} \cdots y_{1}^{-1},\quad \epsilon_{k} = k \mod 2.  \]
  Note that we use \( m < n \) to ensure that \( s_{k} < s_{k-1} \), in which case there is the largest \( k \) for
  which this process works (it is possible that \( k = 1 \)).  For the largest possible \( k \) the block
  \( y_{s_{k}+1} \cdots y_{s_{1}-1} z_{\epsilon_{k}} \) does not cancel completely and so there are
  \( 0 \le p \le s_{k}-1 \) and \( s_{k}+1 \le q \le s_{k-1} \) such that
  \[ \gamma = y_{1} \cdots y_{p} \cdot y_{q} \cdots y_{s_{k-1}-1} z_{\epsilon_{k}} y_{s_{k-1}-1}^{-1} \cdots y_{q}^{-1}
  y_{p+1} \cdots y_{s_{k}-1} z_{1-\epsilon_{k}} y_{s_{k}-1}^{-1} \cdots y_{1}^{-1}, \]
  and this product is ``almost reduced'': it is either reduced or there is exactly one reduction.  If \( q < s_{k-1} \)
  and \( y_{p} \), \( y_{q} \) are not multipliable (including the case \( p=0 \)), then the reduced form of
  \( \gamma \) is
  \[ \gamma = \daggerform{\deta{y_{1} \cdots y_{p}}\dmu{y_{q} \cdots y_{s_{k-1}-1}} z_{\epsilon_{k}}
    \dmui{y_{s_{k-1}-1}^{-1} \cdots y_{q}^{-1}} \dnu{y_{p+1} \cdots y_{s_{k}-1}} z_{1-\epsilon_{k}}
    \dnui{y_{s_{k}-1}^{-1} \cdots y_{p+1}^{-1}} \detai{y_{p}^{-1} \cdots y_{1}^{-1}}}, \]
  and it has the form of \eqref{item:simple}.  If \( q < s_{k-1} \), \( p > 0 \), \( y_{p} \) and \( y_{q} \) are
  multipliable and \( y_{p} \cdot y_{q} \ne e \), then the reduced form of \( \gamma \) is
  \[ \gamma = \daggerform{\deta{y_{1} \cdots y_{p-1}} u \dmu{y_{q+1} \cdots y_{s_{k-1}-1}} z_{\epsilon_{k}}
    \dmui{y_{s_{k-1}-1}^{-1} \cdots y_{q+1}^{-1}} y_{q}^{-1} \dnu{y_{p+1} \cdots y_{s_{k}-1}} z_{1-\epsilon_{k}}
    \dnui{y_{s_{k}-1}^{-1} \cdots y_{p+1}^{-1}} y_{p}^{-1} \detai{y_{p-1}^{-1} \cdots y_{1}^{-1}}},\]
  where \( u = y_{p} \cdot y_{q} \) and \( \gamma \) satisfies item \eqref{item:full}, because
  \( u^{-1}(y_{p}^{-1})^{-1} = y_{q}^{-1}\).  If \( q = s_{k-1} \), then the ``almost reduced'' form of
  \( \gamma \) is
  \[ \gamma = \daggerform{\deta{y_{1} \cdots y_{p}} \cdot z_{\epsilon_{k}} \dnu{y_{p+1} \cdots y_{s_{k}-1}}
    z_{1-\epsilon_{k}} \dnui{y_{s_{k}-1}^{-1} \cdots y_{p+1}^{-1}} \detai{y_{p}^{-1} \cdots y_{1}^{-1}}}. \]
  We see that \( \gamma \) satisfies \eqref{item:simple} if the above product is reduced (including the case \( p = 0
  \)), and finally if \( p \ge 1 \), \( y_{p} \) and \( z_{\epsilon_{k}} \) are multipliable and
  \( y_{p} \cdot z_{\epsilon_{k}} \ne e \), then the reduced form of \( \gamma \) is
  \[ \gamma = \daggerform{\deta{y_{1} \cdots y_{p-1}} u \dnu{y_{p+1} \cdots y_{s_{k}-1}} z_{1-\epsilon_{k}}
    \dnui{y_{s_{k}-1}^{-1} \cdots y_{p+1}^{-1}} y_{p}^{-1} \detai{y_{p-1}^{-1} \cdots y_{1}^{-1}}}, \quad
  u=y_{p} \cdot z_{\epsilon_{k}}, \] which also has the form described by the item \eqref{item:mixed}, because \(
  y_{p}z_{\epsilon_{k}}y_{p}^{-1} = e \) if and only if \( z_{\epsilon_{k}} = e \).

  This exhausts all the possibilities and therefore the lemma is proved.
\end{proof}

\begin{remark}
  \label{rem:description-of-S-squared-is-complete}
  It is easy to see that any element in \( \Gamma \) with the reduced form satisfying one of these items is, in fact, an
  element in \( S^{c}_{\Gamma} \cdot S^{c}_{\Gamma} \) and so the description above is complete, but we shall not use this.
\end{remark}

\begin{corollary}
  \label{cor:for-of-an-asymmetric-element}
  If \( \gamma \in S^{c}_{\Gamma} \cdot S^{c}_{\Gamma} \) has the reduced form \( \gamma = x_{1} \cdots x_{n} \) with
  \( n \ge 2 \) and \( x_{1} \), \( x_{n} \) are not multipliable, then the reduced form of \( \gamma \) can be written
  as \( \mu z_{0} \mu^{-1} \nu z_{1} \nu^{-1} \) for some \( \mu, \nu \in \Gamma \) and some \( z_{0}, z_{1} \in
  X_{\Gamma} \setminus \{e\} \).
\end{corollary}

\begin{lemma}
  \label{lem:S-squared-is-not-Gamma}
  For \( \Gamma = \freeprod \) the inclusion \( S_{\Gamma}^{c} \cdot S_{\Gamma}^{c} \subset \Gamma \) is proper,
  unless \( \Gamma = \mathbb{Z}_{2} \freepr \mathbb{Z}_{2} \).
\end{lemma}
\begin{proof}
  \label{case:element-of-order-bigger-than-two}
  The proof splits into three cases.

  \medskip

  \emph{Case 1}. Suppose there is an \( a \in A \) and a non-trivial \( x \in \Gamma_{a} \) such that the order of
  \( x \) is not two, i.e., \( x \ne x^{-1} \).  Pick \( b \in A\setminus\{a\} \) and a non-trivial
  \( y \in \Gamma_{b} \) and let \( \gamma = (xy)^{3} = xyxyxy \).  We claim that
  \( \gamma \not \in S^{c}_{\Gamma} \cdot S^{c}_{\Gamma} \).  Indeed, if \( \gamma \) were in this set, then by
  Corollary \ref{cor:for-of-an-asymmetric-element} it would have a reduced form
  \( \mu z_{0} \mu^{-1} \nu z_{1} \nu^{-1} \) and using \( x \ne x^{-1} \) it is straightforward to see that this is not
  the case.

  \medskip The remaining cases deal with all \( \Gamma_{a} \) consisting of elements of order \( 2 \).  \medskip

  \emph{Case 2.} Suppose there is \( \Gamma_{a} \) such that \( \card{\Gamma_{a}} \ge 4 \), and pick
  \( b \in A\setminus \{a\} \).  Let \( x_{1}, x_{2}, x_{3} \in \Gamma_{a} \) be distinct non-trivial elements and pick
  a non-trivial \( y \in \Gamma_{b} \).  Using Corollary \ref{cor:for-of-an-asymmetric-element} and since
  \( x_{i} \ne x_{j}^{-1} \) for \( i \ne j \), one sees that
  \( \gamma = x_{1}yx_{2}yx_{3}y \not \in S^{c}_{\Gamma} \cdot S^{c}_{\Gamma} \).

  \medskip

  \emph{Case 3.}  We are left with the case when each of \( \Gamma_{a} \) is isomorphic to \( \mathbb{Z}_{2} \).  Since
  \( \Gamma = \mathbb{Z}_{2} \freepr \mathbb{Z}_{2} \) is excluded in the statement of the theorem, we may assume that
  there are at least three factors.  Let \( x, y, z \) be elements of order \( 2 \) from three different factors and set
  \( \gamma = xyz \).  Note that the reduced form of \( \gamma \) cannot be written as 
  \( \mu z_{0} \mu^{-1} \nu z_{1} \nu^{-1} \) and therefore Corollary \ref{cor:for-of-an-asymmetric-element} finishes
  the proof.
\end{proof}

\begin{remark}
  \label{rem:for-dihedral-S-squared-is-the-whole-group}
  It is easy to check that for the infinite dihedral group \( \mathbb{Z}_{2}\freepr \mathbb{Z}_{2} \) the equality
  \( S^{c}_{\mathbb{Z}_{2}\freepr\mathbb{Z}_{2}} \cdot S^{c}_{\mathbb{Z}_{2}\freepr\mathbb{Z}_{2}} =
  \mathbb{Z}_{2}\freepr\mathbb{Z}_{2} \) indeed holds: \( S^{c}_{\mathbb{Z}_{2}\freepr\mathbb{Z}_{2}} \) consists of all
  elements of odd length plus the identity, so its square is the whole group.
\end{remark}

\section{Completely metrizable free products are discrete}
\label{sec:free-prod-discrete}

In this section the group \( \Gamma = \Gamma_{1} \freepr \Gamma_{2} \) is assumed to have exactly two factors unless
stated otherwise.
\begin{definition}
  \label{def:one-two-type-words}
  For a pair \( (i,j) \in \{1,2\}\times\{1,2\} \) we say that \( \gamma \in \Gamma\setminus\{e\}\) is of type
  \( (i,j) \) if for the reduced form of \( \gamma = x_{1}\cdots x_{n} \) one has \( x_{1} \in \Gamma_{i} \) and
  \( x_{n} \in \Gamma_{j} \).  The type \( (i,j) \) is called \emph{symmetric} if \( i = j \) and \emph{asymmetric}
  otherwise.  Since \( \Gamma = \Gamma_{1} \freepr \Gamma_{2} \) has only two factors, whether an element
  \( \gamma \in \Gamma \) has a symmetric or an asymmetric type depends only upon the parity of its length.  Type
  \( (i,j) \) is said to be the \emph{opposite} of type \( (j,i) \).
\end{definition}

\begin{lemma}
  \label{lem:types-of-non-discrete-sequences}
  Let \( \Gamma \ne \mathbb{Z}_{2} \freepr \mathbb{Z}_{2} \) be a topological group and let \( \gamma_{n} \) be a
  sequence such that \( \gamma_{n} \ne e \), \( \gamma_{n} \to e \) and all the elements \( \gamma_{n} \) have the same
  type.  There is a subsequence \( \gamma_{n_{k}} \) satisfying the following property: for any type \( (i,j) \) there
  is a sequence \( \lambda_{n_{k}} \in \Gamma \), such that \( \lambda_{n_{k}} \to e \), each \( \lambda_{n_{k}} \) has
  type \( (i,j) \), and \( \wlen{\gamma_{n_{k}}} \le \wlen{\lambda_{n_{k}}} \le 4\wlen{\gamma_{n_{k}}} \) for all \( k \).
\end{lemma}
\begin{proof}
  Depending on the type of \( \gamma_{n} \) we have two cases.

  \emph{Case 1: \( \gamma_{n} \) have symmetric type \( (i,i) \).}  If \( x \in \Gamma_{j} \) is non-trivial for
  \( j \ne i \), then elements \( \lambda_{n} = x \gamma_{n} x^{-1} \) have type \( (j,j) \), converge to the identity and
  \( \wlen{\lambda_{n}} = \wlen{\gamma_{n}} + 2 \le 3\wlen{\gamma_{n}} \).  The sequence \( \lambda_{n}\gamma_{n} \) has
  type \( (j,i) \), and \( \gamma_{n}\lambda_{n} \) is of type \( (i,j) \), and elements in both sequences have length
  at most \( 4||\gamma_{n}|| \).  Note that in this
  case we do not have to pass to a subsequence of \( \gamma_{n} \).

  \emph{Case 2: \( \gamma_{n} \) have asymmetric type \( (i,j) \), \( i \ne j \).}  One of \( \Gamma_{1}, \Gamma_{2} \)
  is not isomorphic to \( \mathbb{Z}_{2} \), so let \( \Gamma_{1} \ne \mathbb{Z}_{2} \).  Note that \( \gamma_{n} \) is
  of type \( (i,j) \) implies \( \gamma_{n}^{-1} \) is of the opposite type \( (j,i) \).  So there is no loss in
  generality in assuming \( \gamma_{n} \) is of type \( (1,2) \).

  \emph{Subcase 1: there is \( x \in \Gamma_{1} \) such that for infinitely many \( n \) the first letter of
    \( \gamma_{n} \) is \( x \).} Let \( (n_{k}) \) be such that all the elements \( \gamma_{n_{k}} \) start with
  \( x \) and pick a \( y \in \Gamma_{1} \setminus \{e\} \) such that \( x \cdot y \ne e \), then
  \( \lambda_{n_{k}} = y\gamma_{n_{k}}y^{-1} \) converges to the identity, all the elements \( \lambda_{n_{k}} \) have type
  \( (1,1) \) and \( \wlen{\gamma_{n_{k}}} \le \wlen{\lambda_{n_{k}}} \le 2\wlen{\gamma_{n_{k}}} \).  Similarly to the first
  case one can conjugate \( \lambda_{n_{k}} \) by a non-trivial element from \( \Gamma_{2} \) to get a sequence of type
  \( (2,2) \).
 
  \emph{Subcase 2: for any \( x \in \Gamma_{1} \) there are only finitely many \( \gamma_{n} \) which start with
    \( x \)}.  Let \( x \in \Gamma_{1} \) be any non-trivial element, then \( x\gamma_{n}x^{-1} \) converges to
  the identity, \( \wlen{\gamma_{n}} \le \wlen{x\gamma_{n}x^{-1}} \le 2\wlen{\gamma_{n}} \) and for all but possibly
  finitely many \( n \) elements \( x \gamma_{n} x^{-1} \) have type \( (1,1) \).  Passing to a subsequence to avoid
  this finite set of \( n \)'s we obtain \( \lambda_{n_{k}} \) of type \( (1,1) \), which again can be conjugated by any
  non-trivial element from \( \Gamma_{2} \) into a sequence of type \( (2,2) \).
\end{proof}

\begin{lemma}
  \label{lem:long-words-in-open-sets}
  If \( \Gamma = \Gamma_{1} * \Gamma_{2} \) is a non-discrete metrizable topological group, then the word length
  function is unbounded on any non-empty open set.
\end{lemma}
\begin{proof}
  If \( \Gamma \ne \mathbb{Z}_{2} \freepr \mathbb{Z}_{2} \), then by Lemma \ref{lem:types-of-non-discrete-sequences} for
  any open set \( U \) and any \( \gamma \in U \) we can find \( \lambda \ne e \) such that the product
  \( \lambda \gamma \) is reduced and \( \lambda \gamma \in U \), so
  \( \wlen{\lambda\gamma} = \wlen{\lambda}+\wlen{\gamma} \), and thus \( \wlen{\cdot} \) is unbounded on \( U \).

  For \( \Gamma = \mathbb{Z}_{2} \freepr \mathbb{Z}_{2} \) the statement is obvious, since for such a \( \Gamma \) the
  set of elements of length at most \( n \) is finite for any \( n \) and \( \Gamma \) being non-discrete implies finite
  sets cannot be open.
\end{proof}

Recall that a subset of a metric space is called \emph{separated} if there is a positive lower bound on
the distance between its distinct points.

\begin{lemma}
  \label{lem:S-squared-has-empty-interior}
  If \( \Gamma  = \Gamma_{1} * \Gamma_{2} \) is completely metrizable and non-discrete, then the interior of
  \( S^{c}_{\Gamma} \cdot S^{c}_{\Gamma} \) is empty.
\end{lemma}

\begin{proof}
  First of all \( \Gamma \) being completely metrizable and non-discrete must be uncountable, so there is no need to
  worry about \( \Gamma \) being isomorphic to the infinite dihedral group.  Fix a left-invariant metric \( d \) on
  \( \Gamma \) and set \( L_{n} = \{ \gamma \in \Gamma: \wlen{\gamma} \le n \} \), \( \Gamma = \bigcup_{n} L_{n} \).
  Since a separated set is necessarily closed and nowhere dense (by non-discreteness of the group), Baire's theorem
  ensures the existence of \( N \) such that \( L_{N} \) is not separated: there are sequences
  \( \gamma_{n}', \gamma_{n}'' \in L_{N} \) such that \( \gamma'_{n} \ne \gamma''_{n} \) for all \( n \) and
  \( d(\gamma_{n}',\gamma_{n}'') \to 0 \).  By left-invariance of \( d \) we see that
  \( \gamma_{n} = (\gamma_{n}'')^{-1}\gamma_{n}' \) converges to \( e \), and \(\gamma_{n} \in L_{2N} \).  By passing to
  a subsequence and using Lemma \ref{lem:types-of-non-discrete-sequences} we may assume that \( M \) is such that for
  any type \( (i,j) \) there is a sequence \( \bar{\gamma}_{n} \in L_{M} \), \( \bar{\gamma}_{n} \to e \) that consists
  of elements of type \( (i,j) \).
 
  Pick a non-empty open \( U \); we shall find an element in \( U \) but not in
  \( S^{c}_{\Gamma} \cdot S^{c}_{\Gamma} \). By Lemma \ref{lem:long-words-in-open-sets} find \( \lambda \in U \) with
  the reduced form \( \lambda = x_{1} \cdots x_{l} \), \( l \ge 8M+2 \).  We now construct inductively elements
  \( \gamma_{k} \) for \( k=1, \ldots, l \) such that for
  \[ \lambda_{k} = x_{1}\gamma_{1}x_{2}\gamma_{2} \cdots x_{k}\gamma_{k}x_{k+1}x_{k+2}\cdots x_{l} \]
  and for all \( k \) we have
  \begin{enumerate}[(i)]
  \item the product in the definition of \( \lambda_{k} \) is reduced;
  \item \( \lambda_{k} \in U \);
  \item \( \wlen{\gamma_{k}} \le M \);
  \item\label{item:new-letters} if \( W_{k} \) denotes the set of letters in the reduced form of \( \lambda_{k} \), then
    there is a letter \( y_{k} \) in the reduced form of \( \gamma_{k} \) such that \( y_{k} \not \in W_{k-1}^{-1} \);
  \item\label{item:last-step} the last element \( \lambda_{l} \) has an asymmetric type \( (1,2) \) or \( (2,1) \).
  \end{enumerate}
  The first step of the construction is not any different from the later steps, so we show how to find \( \gamma_{k} \).
  Suppose that \( \gamma_{1},\ldots,\gamma_{k-1} \) have been constructed.  By the choice of \( L_{M} \) we can find a
  sequence \( \bar{\gamma}_{n} \to e \), such that \( \wlen{\bar{\gamma}_{n}} \le M \) and
  \( x_{k}\bar{\gamma}_{n}x_{k+1} \) is reduced for all \( n \) (and we take \( x_{l+1} \) to be \( x_{1} \) in the last
  step to satisfy item \eqref{item:last-step}).  Throwing away finitely elements from the sequence we may assume that
  \( x_{1}\gamma_{1} \cdots x_{k-1}\gamma_{k-1}x_{k}\bar{\gamma}_{n}x_{k+1}\cdots x_{l} \in U \) for all \( n \).  Since
  \( W_{k-1}^{-1} \) is finite and since there are only finitely many words in the alphabet \( W_{k-1}^{-1} \) of length
  at most \( M \), there must be \( \bar{\gamma}_{n_{0}} \) that has a letter not in \( W_{k-1}^{-1} \) in its reduced
  form; we set \( \gamma_{k} = \bar{\gamma}_{n_{0}} \).

  We finally claim that \( \lambda_{l} \not \in S^{c}_{\Gamma} \cdot S^{c}_{\Gamma} \).  Indeed, suppose
  \( \lambda_{l} \) were in this set.  Since \( \lambda_{l} \) is of asymmetric type, Corollary
  \ref{cor:for-of-an-asymmetric-element} implies the reduced form of \( \lambda_{l} \) can be written as
  \( \mu z_{0} \mu^{-1} \nu z_{1} \nu^{-1}\).  Assume for definiteness that \( \wlen{\mu} \ge \wlen{\nu} \), then
  \( \wlen{\mu} \ge \frac{8M+2-2}{4} = 2M \) and therefore \( \wlen{\gamma_{k}} \le M \) implies there is
  \( 1 \le k \le l \) such that \( \gamma_{k} \subseteq \mu^{-1} \), but then every letter in \( \gamma_{k} \) is also
  in \( W_{k-1}^{-1}\), contradicting the construction.

  Since \( \lambda_{l} \in U \), we have proved that
  \( U \not \subseteq S^{c}_{\Gamma} \cdot S^{c}_{\Gamma} \), and since \( U \) was arbitrary, the conclusion follows.
\end{proof}

\begin{theorem}
  \label{thm:free-products-are-not-completely-metrizable}
  Any completely metrizable group topology on \( \Gamma \) is discrete.
\end{theorem}
\begin{proof}
  Suppose the statement is false and \( \Gamma \) is equipped with a non-discrete completely metrizable group topology.
  Let \( F_{n} = \{ f \in \Gamma : f^{n} = e\} \), and therefore \( F_{\Gamma} = \bigcup_{n} F_{n} \), where each of
  \( F_{n} \) is closed.  Recall that \( F_{\Gamma} \subseteq S^{c}_{\Gamma} \).

  Fix a complete metric \( d \) on \( \Gamma \).  We assume that \( d \le 1 \).  As in the original construction of
  Dudley we build a sequence \( g_{n} \in \Gamma \) such that for \( r_{n} = n + \sum_{k=1}^{n}\wlen{g_{k}} \) and
  \[ \widetilde{h}_{n,m} = g_{m}(g_{m+1}( \cdots (g_{n-1}(g_{n})^{r_{n-1}})^{r_{n-2}} \cdots)^{r_{m+1}})^{r_{m}}, \quad
  1 \le m \le n, \] we have
  \begin{enumerate}[(i)]
  \item \( g_{n} \not \in S^{c}_{\Gamma} \cdot S^{c}_{\Gamma} \);
  \item for all \( m \ge 2 \)
    \[ d(\widetilde{h}_{n+1,m},\widetilde{h}_{n,m}) < 2^{-n}d(g_{m},F_{r_{m-1}}), \]
    and \( d(\widetilde{h}_{n+1,1},\widetilde{h}_{n,1}) < 2^{-n} \).
  \end{enumerate}
  Such a sequence is constructed by induction on \( n \).  For the base of induction choose any
  \( g_{1} \not \in S^{c}_{\Gamma} \cdot S^{c}_{\Gamma} \).  For the induction step note that for each
  \( 1 \le m \le n \) the map
  \[ \xi_{m}(x) = g_{m}(g_{m+1}( \cdots (g_{n-1}(g_{n} \cdot x^{r_{n}})^{r_{n-1}})^{r_{n-2}}
  \cdots)^{r_{m+1}})^{r_{m}}. \]
  is continuous and \( \xi_{m}(e) = \widetilde{h}_{n,m} \).  Therefore there exists an open neighbourhood of the identity
  \( U \) such that for all \( g \in U \) and all \( 2 \le m \le n \)
  \[ d(\xi_{m}(g),\widetilde{h}_{n,m}) < 2^{-n} d(g_{m},F_{r_{m-1}}) \]
  and \( d(\xi_{1}(g),\widetilde{h}_{n,1}) < 2^{-n} \).  By Lemma \ref{lem:S-squared-has-empty-interior} the interior of
  \( S^{c}_{\Gamma} \cdot S^{c}_{\Gamma} \) is empty and we choose \( g_{n+1} \in U \) to be any element not in
  \( S^{c}_{\Gamma} \cdot S^{c}_{\Gamma} \).  This finishes the step of induction.

  For each \( m \) the resulting sequence \( \widetilde{h}_{n,m} \) is Cauchy, so we may define
  \( h_{m} = \lim_{n \to \infty} \widetilde{h}_{n,m} \).  Also by the definition of \( \widetilde{h}_{n,m} \) we have
  \( \widetilde{h}_{n,m} = g_{m}\widetilde{h}_{n,m+1}^{r_{m}} \) and therefore also \( h_{m} = g_{m}h_{m+1}^{r_{m}} \).
  Finally \( d(\widetilde{h}_{n+1,m},\widetilde{h}_{n,m}) < 2^{-n}d(g_{m}, F_{r_{m-1}}) \) implies
  \[ d(h_{m},g_{m}) < \sum_{n=m}^{\infty} 2^{-n}d(g_{m},F_{r_{m-1}}) \le d(g_{m},F_{r_{m-1}}), \]
  and therefore \( h_{m} \not \in F_{r_{m-1}} \), whence by item \eqref{item:when-power-is-not-shorter} of Lemma
  \ref{lem:description-of-S-and-F} we see that \( \wlen{h_{m}^{r_{m-1}}} \ge \wlen{h_{m}} \) holds for all \( m \ge 2
  \).

  \medskip

  We claim that \( \wlen{h_{m}^{r_{m-1}}} \ge r_{m-1} \) for all \( m \ge 2 \).  Indeed, if
  \( h_{m} \in S_{\Gamma} \), then \( \wlen{h_{m}^{r_{m-1}}} \ge r_{m-1} \) holds by the definition of \( S_{\Gamma} \).
  If \( h_{m} \in S^{c}_{\Gamma} \), but \( h_{m+1} \in S_{\Gamma} \), then
  \[ \wlen{h_{m}^{r_{m-1}}} \ge \wlen{h_{m}} = \wlen{g_{m}h_{m+1}^{r_{m}}} \ge \wlen{h_{m+1}^{r_{m}}} - \wlen{g_{m}} \ge
  r_{m} - \wlen{g_{m}} = r_{m-1} + 1 > r_{m-1}. \]
  Finally if \( h_{m} \in S^{c}_{\Gamma} \) and \( h_{m+1} \in S^{c}_{\Gamma} \), then by item \eqref{item:description-S}
  of Lemma \ref{lem:description-of-S-and-F} \( h_{m+1}^{-r_{m}} \in S^{c}_{\Gamma} \) and therefore
  \( g_{m} = h_{m}h_{m+1}^{-r_{m}} \in S^{c}_{\Gamma}\cdot S^{c}_{\Gamma} \) contradicting the choice of \( g_{m} \).

  \medskip

  We are now ready to arrive at a contradiction by showing that \( \wlen{h_{1}} \) is unbounded:
  \begin{displaymath}
    \begin{aligned}
      \wlen{h_{1}} &= \wlen{g_{1}h_{2}^{r_{1}}} \ge \wlen{h_{2}^{r_{1}}} - \wlen{g_{1}} \ge \wlen{h_{2}} - \wlen{g_{1}}\\
      &= \wlen{g_{2}h_{3}^{r_{2}}} - \wlen{g_{1}} \ge \wlen{h_{3}^{r_{2}}} - \wlen{g_{1}} - \wlen{g_{2}} \ge
      \wlen{h_{3}} - \wlen{g_{1}} -
      \wlen{g_{2}}\\[9pt]
      &=\quad \cdots \\
      &= \wlen{g_{m}h_{m+1}^{r_{m}}} - \sum_{l=1}^{m-1} \wlen{g_{l}} \ge \wlen{h_{m+1}^{r_{m}}} -
      \sum_{l=1}^{m}\wlen{g_{l}} \ge r_{m} - \sum_{l=1}^{m}\wlen{g_{l}} = m.
    \end{aligned}
  \end{displaymath}
  And this is true for all \( m \), which is absurd.
\end{proof}

\begin{corollary}
  \label{cor:completely-metrizable-free-products-discrete-general}
  Let \( \Gamma = \freeprod \), where \( A \) is any set of cardinality at least
  two.  Any completely metrizable group topology on \( \Gamma \) is discrete.
\end{corollary}
\begin{proof}
  This is immediate from the above, since
  \( \freeprod = \Gamma_{b} \freepr \Bigl(\fp{a}{A\setminus\{b\}}{\Gamma}\Bigr) \).
\end{proof}

\begin{corollary}
  \label{cor:closed-kernel-implies-continuity}
  Let \( \phi : G \to \Gamma \) be a surjective homomorphism from a completely metrizable topological group \( G \) onto
  a discrete free product \( \Gamma = \freeprod \) for some index set \( \card{A} \ge 2 \).  If the kernel \( K \) of
  \( \phi \) is closed in \( G \), then \( \phi \) is continuous.
\end{corollary}
\begin{proof}
  Suppose that the kernel \( K \) is a closed normal subgroup of \( G \) and let \( \pi : G \to G/K \) be the canonical
  projection.  The surjective homomorphism \( \phi \) induces the isomorphism \( \psi : G/K \to \Gamma \).  Since the
  group \( G/K \) being a quotient of a completely metrizable group by a closed normal subgroup is itself completely
  metrizable in the quotient topology (see \cite[Theorem 5.2.2]{MR1425877} and \cite[Section 8D]{MR1321597}), Corollary
  \ref{cor:completely-metrizable-free-products-discrete-general} implies \( G/K \) is discrete.  Therefore
  \( \psi : G/K \to \Gamma \) is continuous, and hence so is \( \phi = \psi \circ \pi \).
\end{proof}

\section{More calculations in the free products}
\label{sec:more-calc-free}

To motivate some of the calculations in this section let us discuss the strategy of proving our main result Theorem
\ref{thm:continuity-into-free-products}.  Based on Corollary \ref{cor:closed-kernel-implies-continuity} and some
additional arguments we shall find ourselves in the situation of having a surjective homomorphism
\( \phi : G \to \Gamma \) from a completely metrizable group into a free product that has a \emph{dense} kernel.
Density of the kernel gives a great power.  Recall that in the argument of Dudley we could choose the sequence
\( g_{m} \) from any dense set.  In particular, for such a \( \phi \) the sequence \( \phi(g_{m}) \) can be literally
any sequence we want.  For instance, we shall find it convenient to have \( \phi(g_{m}) = \gamma_{0} \in \Gamma \) for a
concrete \( \gamma_{0} \in \Gamma \).  To obtain a contradiction we shall need to show that
\( \phi(h_{m}) \not \in S^{c}_{\Gamma} \), where \( h_{m} \) is constructed as usually.  Since elements \( h_{m} \) are
obtained via the limit procedure, we shall have little control of them.  The only piece of information that will
be at our hands is the equality \( h_{m} = g_{m} h_{m+1}^{r_{m}} \), and therefore also
\( \phi(h_{m}) = \gamma_{0} \phi(h_{m+1})^{r_{m}} \).  If we choose
\( \gamma_{0} \not \in S^{c}_{\Gamma} \cdot S^{c}_{\Gamma} \), then either \( \phi(h_{m}) \) or \( \phi(h_{m+1}) \) has
to be in \( S_{\Gamma} \), but we need this to be the case for both of them.  The somewhat tedious calculations in this
section show that under some mild assumptions on \( \gamma_{0} \) and \( r_{m} \),
\( \gamma_{0}\lambda^{r_{m}} \in S^{c}_{\Gamma} \) implies \( \lambda \in S^{c}_{\Gamma} \) and this will be enough for
our purposes.

In this section the group \( \Gamma = \Gamma_{1} \freepr \Gamma_{2} \) is also assumed to have exactly two factors
unless stated otherwise.  The role of the groups \( \Gamma_{1} \) and \( \Gamma_{2} \) will not be symmetric.  As we
have noted earlier, elements of order two cause some troubles, and we shall use two different arguments: when both
\( \Gamma_{1} \) and \( \Gamma_{2} \) ``essentially'' consist  of elements of order two, and when one of \( \Gamma_{1}
\), \( \Gamma_{2} \) has ``many'' elements of order bigger than two.  Results of this section are needed for the latter
and will be applied for \( \Gamma_{1} \) being a ``good'' group in this sense.

\begin{definition}
  \label{def:balanced-word}
  Let \( \gamma \) be an element in \( \Gamma_{1} \freepr \Gamma_{2} \) with the reduced form
  \( \gamma = x_{1}x_{2}\cdots x_{n} \).  If \( n \ge 2 \), then \( \gamma \) has letters from both \( \Gamma_{1} \) and
  \( \Gamma_{2} \), so let \( i_{1}, \ldots, i_{m} \) be the indices of letters from \( \Gamma_{1} \), i.e., depending
  on whether \( x_{1} \in \Gamma_{1} \) or \( x_{1} \in \Gamma_{2} \), \( i_{k} \)'s are either all the odd or all the
  even letters.  The word \( \gamma \) is called \emph{unbalanced} if \( x_{i_{k}} \ne x_{i_{l}}^{\pm 1} \) for all
  \( k \ne l \) and \( x_{i_{k}} \ne x_{i_{k}}^{-1} \) for all \( k \).  In other words, an element is unbalanced if all
  of its letters from \( \Gamma_{1} \) have order bigger than two, and the word uses neither the same letter from
  \( \Gamma_{1} \) twice nor the letter from \( \Gamma_{1} \) and its inverse.
\end{definition}

\begin{lemma}
  \label{lem:asymmetric-unbalanced-is-not-in-Scomp-squared}
  An unbalanced asymmetric element of length at least \( 6 \) does not belong to \( S^{c}_{\Gamma} \cdot S^{c}_{\Gamma} \).
\end{lemma}

\begin{proof}
  Let \( \gamma \in \Gamma \) be such an element.  By Corollary \ref{cor:for-of-an-asymmetric-element} if \( \gamma \in
  S^{c}_{\Gamma} \cdot S^{c}_{\Gamma} \), then the reduced form of \( \gamma \) is of the form \( \mu z_{0} \mu^{-1} \nu
  z_{1} \nu^{-1} \) for some \( \mu, \nu \in \Gamma \) and some \( z_{0}, z_{1} \in X_{\Gamma} \setminus \{e\} \).
  Since \( \gamma \) is unbalanced, both \( \mu \) and \( \nu \) may consist only of letters from \( \Gamma_{2} \).  In
  particular, \( z_{0} \) and \( z_{1} \) are the only possible letters from \( \Gamma_{1} \), but
  \( ||\gamma|| \ge 6 \) and so \( \gamma \) has at least \( 3 \) letters from \( \Gamma_{1} \) in its reduced form.
  This contradiction proves the lemma.
\end{proof}

\begin{lemma}
  \label{lem:balanced-times-power-of-reduced-not-is-in-S}
  Let \( \xi \in \Gamma_{1} \freepr \Gamma_{2} \) be an unbalanced element of asymmetric type and of length
  \( \wlen{\xi} \ge 6 \).  If \( r \ge 4 \) and \( \lambda \in \Gamma_{1} \freepr \Gamma_{2} \) is cyclically reduced of
  length \( \wlen{\lambda} \ge 2 \), then \( \xi \lambda^{r} \not \in S^{c}_{\Gamma} \).
\end{lemma}

\begin{proof}
  The proof is by contradiction: suppose the statement is false and we are given elements
  \( \xi, \lambda \in \Gamma_{1} \freepr \Gamma_{2} \) and \( r \) which comply with the assumptions of the theorem, but
  \( \xi \lambda^{r} \in S^{c}_{\Gamma}\).  Therefore \( \xi \lambda^{r} = \beta z \beta^{-1} \) for some
  \( \beta \in \Gamma_{1} \freepr \Gamma_{2} \) and \( z \in \Gamma_{1} \cup \Gamma_{2} \).  A tiresome case analysis
  with numerous subcases will show that this is impossible.  More precisely, we consider four cases depending on the
  relation between the types of \( \xi \) and \( \lambda \) and there will be subcases dealing with different possible
  lengths of \( \lambda \).  Most of the products in this proof will be reduced (except for the product
  \( \xi \lambda^{r} \), the reduced form of which we aim to understand), we shall mention explicitly when a product
  under consideration is not reduced and a dot \( \cdot \) will denote places where reductions may occur.

  In this proof we shall consistently use the following notations.  Symbols \( x \) and \( y \) will denote letters in
  the reduced form of \( \xi \), while symbols with the tilde, e.g., \( \tilde{x} \) and \( \tilde{y} \),
  will denote letters in the reduced form of \( \lambda \).  Symbols \( u \) and \( v \) will denote non-trivial
  products of two non-trivial multipliable letters.  A typical contradiction will be obtained by showing that such a
  product is equal to one of its factors.  For example, for \( u = x \cdot \tilde{x} \) with  \( x, \tilde{x} \ne e \)
  we shall sometimes get \( u = \tilde{x} \), which is clearly impossible.

  Let \( \xi \lambda^{r} = \beta z \beta^{-1} \) and so, roughly speaking, the first half of the letters in the reduced
  form of \( \xi \lambda^{r} \) are just the inverses of the corresponding letters in the second half.  We analyse the
  product \( \xi \lambda^{r} \) and show that it cannot have the form \( \beta z \beta^{-1} \).  To emphasise at which
  parts of the reduced form of \( \xi \lambda^{r} \) we look in the concrete argument, we under-bracket the relevant
  pieces.  By assumption \( \xi \) has type \( (i,j) \) with \( i \ne j \).

  \bigskip \emph{Case: \( \xi \) is of type \( (i,j) \) and \( \lambda \) has type \( (j,i) \).}  In this case
  \( \lambda \) is of asymmetric type, and the reduced form of \( \lambda^{r} \) is just the concatenation of \( r \)
  copies of the reduced form of \( \lambda \).  Note also that the length of \( \lambda \) is even.

  \medskip

  \emph{Subcase: \( \wlen{\lambda} \ge \wlen{\xi} + 2 \).}  If \( \lambda = \xi^{-1}\lambda_{1} \) for some
  \( \lambda_{1} \in \Gamma \), then
  \[ \xi\lambda^{r} = \symform{\lpart{\lambda_{1}\xi^{-1}}\lambda_{1}\xi^{-1}\lambda_{1} \cdots
    \rpart{\xi^{-1}\lambda_{1}}}, \]
  and the product is reduced (we shall not mention this in the future). Therefore for \( \xi \lambda^{r} \) to have the
  form \( \beta z \beta^{-1} \) we must have \( \lambda_{1}^{-1} = \lambda_{1} \) and \( \xi^{-1} = \xi \), whence
  \( \xi \) has a symmetric type contrary to the original assumption.

  If \( \lambda \ne \xi^{-1}\lambda_{1} \), then there is a unique way to decompose \( \xi \) as \( \xi_{1} x \xi_{2} \)
  and \( \lambda \) as \( \xi_{2}^{-1} \tilde{x} \lambda_{1} \) for some \( \xi_{1}, \xi_{2}, \lambda_{1} \in \Gamma \)
  and some \( x, \tilde{x} \in X_{\Gamma} \), where \( x, \tilde{x} \ne e \), \( x \) and \( \tilde{x} \) are
  multipliable, \( \wlen{\lambda_{1}} \ge \wlen{\xi_{1}} + 2 \) and \( x \cdot \tilde{x} \ne e \).  Note that we allow
  for \( \xi_{1} \) and \( \xi_{2} \) to be empty.  So
  \[ \xi\lambda^{r} = \symform{\lpart{\xi_{1}u}\lambda_{1}\xi_{2}^{-1}\tilde{x}\lambda_{1} \cdots
    \xi_{2}^{-1}\tilde{x}\rpart{\lambda_{1}}},\quad u= x\cdot \tilde{x}. \]
  This implies \( \lambda_{1} = \lambda_{2}u^{-1}\xi_{1}^{-1} \), \( \wlen{\lambda_{2}} \ge 1 \) and thus
  \[ \xi\lambda^{r} = \symform{\xi_{1}u \lpart{\lambda_{2}u^{-1}} \xi_{1}^{-1 }\xi_{2}^{-1} \tilde{x}\lambda_{2} u^{-1}
    \xi_{1}^{-1} \cdots \xi_{2}^{-1}\tilde{x}\lambda_{2} u^{-1} \xi_{1}^{-1} \xi_{2}^{-1} \rpart{\tilde{x}\lambda_{2}} u^{-1}
    \xi_{1}^{-1}}. \]
  So \( \lambda_{2} = \lambda_{2}^{-1} \) and \( u = \tilde{x} \).  But \( u = x \cdot \tilde{x} \), hence \( x = e \),
  which is impossible.

  \medskip

  \emph{Subcase: \( ||\xi|| = ||\lambda|| \).}  If \( \lambda = \xi^{-1} \), then \( \xi \lambda^{r} = \lambda^{r-1} \)
  has asymmetric type and \( \xi \lambda^{r} \not \in S^{c}_{\Gamma} \).  So we may assume that
  \( \xi = \xi_{1} x \xi_{2} \) and \( \lambda = \xi_{2}^{-1}\tilde{x} \lambda_{1} \), \( x \) and \( \tilde{x} \) are
  multipliable letters, \( x\cdot \tilde{x} \ne e \) and \( ||\lambda_{1}|| = ||\xi_{1}|| \).  Similarly to the previous
  subcase one sees that \( \lambda_{1} = \xi_{1}^{-1} \) and so
  \[ \xi \lambda^{r} = \symform{\xi_{1} \lpart{u \xi_{1}^{-1}\xi_{2}^{-1}}\tilde{x}\xi_{1}^{-1} \cdots
  \xi_{2}^{-1}\tilde{x} \rpart{\xi_{1}^{-1}\xi_{2}^{-1}\tilde{x}}\xi_{1}^{-1}}, \quad u = x \cdot \tilde{x}.  \]
  Thus \( u^{-1} = \tilde{x} \) and \( \xi_{2}\xi_{1} = \xi_{1}^{-1}\xi_{2}^{-1} \).    But
  the latter is impossible, since \( ||\xi_{2}\xi_{1}|| = ||\xi|| - 1 \ge 5 \) implies there is a letter from
  \( \Gamma_{1} \) in \( \xi_{2}\xi_{1} \) that is also a letter in the reduced form of \( \xi_{1}^{-1}\xi_{2}^{-1} \).
  This contradicts the assumption that \( \xi \) is unbalanced.

  \medskip

  \emph{Subcase: \( 4 \le | |\lambda|| \le ||\xi|| -2 \).}

  \emph{Subsubcase: \( \xi = \xi_{1} x \lambda^{-1} \), \( \lambda = \tilde{x}\lambda_{1} \)
    and \( x \cdot \tilde{x} \ne e\).}  In this situation \( ||\xi_{1}|| \ge 1 \), \( ||\lambda_{1}|| \ge 3 \) and
  \[ \xi \lambda^{r} = \symform{\lpart{\xi_{1}}u \lambda_{1}\tilde{x}\lambda_{1} \cdots \tilde{x}\rpart{\lambda_{1}}},
  \quad u = x \cdot \tilde{x}. \]
  If either \( ||\xi_{1}|| \ge 2 \) or \( ||\xi_{1}||=1 \) and
  \( \xi_{1} \in \Gamma_{1} \), then \( \xi_{1} \) contains a letter from \( \Gamma_{1} \) with its inverse being a
  letter in \( \lambda_{1} \), but \( \xi = \xi_{1}x\lambda_{1}^{-1}\tilde{x}^{-1} \), and in particular \( \xi \)
  contains two copies of the same letter from \( \Gamma_{1} \) contradicting the assumption that \( \xi \) is
  unbalanced.  Therefore \( ||\xi_{1}|| = 1 \), \( \xi_{1} \in \Gamma_{2} \) and thus
  \( \lambda_{1} = \lambda_{2}u^{-1}\xi_{1}^{-1} \) with \( ||\lambda_{2}|| \ge 1\),
  \[ \xi \lambda^{r} = \symform{\xi_{1}u \lpart{\lambda_{2} u^{-1}} \xi_{1}^{-1} \tilde{x}\lambda_{2}u^{-1}\xi_{1}^{-1}
    \cdots \rpart{\tilde{x}\lambda_{2}}u^{-1}\xi_{1}^{-1}}. \]
  Therefore \( \lambda_{2}^{-1} = \lambda_{2} \) and \( u = \tilde{x} \).  But \( u = x \cdot \tilde{x} \), hence \( x =
  e \), which is impossible.

  \emph{Subsubcase: \( \xi = \xi_{1} x \lambda^{-1} \), \( \lambda = \tilde{x}\lambda_{1} \)
    and \( x \cdot \tilde{x} = e\).}  Letters \( x \) and \( \tilde{x} \) cannot belong to \( \Gamma_{1} \), since this
  would contradict \( \xi \) being unbalanced, so \( x = \tilde{x}^{-1} \in \Gamma_{2} \).  Let
  \( \xi_{1} = \xi_{2}y \), \( \lambda = x^{-1}\tilde{y}\lambda_{2} \) and we must have \( y\cdot \tilde{y} \ne e \)
  (again, because \( \xi \) is unbalanced), \( ||\xi_{2}|| \ge 0 \),
  \[ \xi\lambda^{r} = \symform{\lpart{\xi_{2}}v \lambda_{2}x^{-1}\tilde{y}\lambda_{2} \cdots
    x^{-1}\tilde{y}\rpart{\lambda_{2}}}, \quad v = y\cdot \tilde{y}.  \]
  Note that \( ||\lambda|| \ge 4 \) implies \( ||\lambda_{2}|| \ge 2 \).  If \( ||\xi_{2}|| \ge 2 \), then there exists
  a letter in \( \xi_{2} \) from \( \Gamma_{1} \) with its inverse being a letter in \( \lambda_{2} \).  So
  \( \xi = \xi_{2} y x \lambda_{2}^{-1} \tilde{y}^{-1}x \) has two copies of a letter from \( \Gamma_{1} \), which
  is impossible.  Thus \( ||\xi_{2}|| < 2 \), but \( ||\xi_{2}|| \) is even, and therefore \( ||\xi_{2}|| = 0 \), in
  other words
  \[ \xi\lambda^{r} = \symform{\lpart{v} \lambda_{2}x^{-1}\tilde{y}\lambda_{2} \cdots x^{-1}\tilde{y}\rpart{\lambda_{2}}}. \]
  So \( \lambda_{2} = \lambda_{3} v^{-1} \),
  \[ \xi\lambda^{r} = \symform{v \lpart{\lambda_{3} v^{-1 }}x^{-1}\tilde{y}\lambda_{3} v^{-1} \cdots
    x^{-1}\rpart{\tilde{y} \lambda_{3}}v^{-1}}, \]
  whence \( \lambda_{3}^{-1} = \lambda_{3} \) and \( v= \tilde{y} \).  But \( v = y \cdot \tilde{y} \), implying \( y =
  e \), which is impossible.

  \medskip

  Therefore the reduced form of \( \xi \) cannot be written in the form \( \xi = \xi_{1} x \lambda^{-1} \), and we have
  one subsubcase left.

  \emph{Subsubcase: \( \xi \) cannot be written as \( \xi_{1} x \lambda^{-1} \).}  In this subsubcase we can
  write \( \xi = \xi_{1} x \xi_{2} \) and \( \lambda = \xi_{2}^{-1} \tilde{x}\lambda_{1} \), for some
  \( \xi_{1}, \xi_{2}, \lambda_{1} \in \Gamma \) and some \( x, \tilde{x} \in X_{\Gamma} \setminus\{e\} \),
  \( x \cdot \tilde{x} \ne e \), \( ||\xi_{1}|| \ge ||\lambda_{1}|| + 2 \),
  \[ \xi \lambda^{r} = \symform{\lpart{\xi_{1}}u \lambda_{1}\xi_{2}^{-1}\tilde{x}\lambda_{1} \cdots
    \xi_{2}^{-1}\rpart{\tilde{x}\lambda_{1}}}, \quad u = x \cdot \tilde{x}.  \]
  Therefore \( \xi_{1} = \lambda_{1}^{-1}\tilde{x}^{-1}\xi_{3} \), \( ||\xi_{3}|| \ge 1 \),
  \[ \xi \lambda^{r} = \symform{\lambda_{1}^{-1} \tilde{x}^{-1} \lpart{\xi_{3}} u \lambda_{1} \xi_{2}^{-1}\tilde{x}
    \lambda_{1} \cdots \xi_{2}^{-1}\tilde{x} \rpart{\lambda_{1} \xi_{2}^{-1}} \tilde{x} \lambda_{1}}.\]
  Similarly to the earlier cases, since \( \xi = \lambda_{1}^{-1} \tilde{x}^{-1} \xi_{3} x \xi_{2} \) is unbalanced and
  \( ||\lambda_{1}\xi_{2}^{-1}|| \ge 3 \) (because \( ||\lambda|| \ge 4 \)) we must have \( ||\xi_{3}|| = 1 \) and
  \( \xi_{3} \in \Gamma_{2} \).  So \( \lambda_{1}\xi_{2}^{-1} = \zeta u^{-1} \xi_{3}^{-1} \) for some
  \( \zeta \in \Gamma \), \( ||\zeta|| \ge 1 \) and
  \[ \xi \lambda^{r} = \symform{\lambda_{1}^{-1} \tilde{x}^{-1} \xi_{3} u \lpart{\zeta u^{-1}} \xi_{3}^{-1 } \tilde{x}
    \cdots \rpart{\tilde{x} \zeta} u^{-1} \xi_{3}^{-1 } \tilde{x} \lambda_{1}}.\]
  Finally we get \( \zeta = \zeta^{-1} \) and \( u = \tilde{x} \), contradicting, as before,  \( x \ne e \).

  \medskip

  \emph{Subcase: \( ||\lambda|| = 2 \)}.
  
  \emph{Subsubcase: \( \xi = \xi_{1}x\lambda^{-1} \), \( \lambda = \tilde{x}\tilde{y} \), \( x \cdot \tilde{x} \ne e
    \).}  Here \( ||\xi_{1}|| \ge 3 \) and
  \[ \xi \lambda^{r} = \symform{\lpart{\xi_{1}} u \tilde{y} \tilde{x} \tilde{y} \cdots \rpart{\tilde{x}\tilde{y}
    \tilde{x}\tilde{y}}},\quad u = x \cdot \tilde{x}.  \]
  If \( ||\xi_{1}|| \ge 4 \) or if \( \tilde{y} \in \Gamma_{1} \), then \( \xi_{1} \) contains two copies of the same
  letter from \( \Gamma_{1} \), which is impossible by assumption.  So
  \( \xi_{1} = \tilde{y}^{-1} \tilde{x}^{-1} \tilde{y}^{-1} \) and \( \tilde{y} \in \Gamma_{2} \).
  Therefore \( \xi = \tilde{y}^{-1} \tilde{x}^{-1} \tilde{y}^{-1} x \tilde{y}^{-1} \tilde{x}^{-1} \) and such an element
  is not unbalanced.

  \emph{Subsubcase: \( \xi = \xi_{1}x\lambda^{-1} \), \( \lambda = \tilde{x}\tilde{y} \), \( x \cdot \tilde{x} = e \).}
  For an unbalanced \( \xi \) the equality \( x \cdot \tilde{x} = e \) is only possible when
  \( x, \tilde{x} \in \Gamma_{2} \) and so \( \xi = \xi_{2} y x \lambda^{-1}\), \( \lambda = x^{-1}\tilde{y} \) and
  \( y \cdot \tilde{y} \ne e \), because \( \xi \) is unbalanced.  So
  \[ \xi \lambda^{r} = \symform{\lpart{\xi_{2}} v x^{-1} \tilde{y} \cdots x^{-1} \tilde{y} x^{-1} \rpart{\tilde{y}
      x^{-1} \tilde{y}}}, \quad v = y \cdot \tilde{y}, \]
  and \( \tilde{y} \in \Gamma_{1} \).  Similarly to the previous subsubcase we get \( ||\xi_{2}|| = 2 \), and
  \( \xi_{2} = \tilde{y}^{-1} x\), and therefore
  \( \xi = \tilde{y}^{-1} x y x \tilde{y}^{-1} x \), which again contradicts the assumption that \( \xi \) is
  unbalanced.

  \emph{Subsubcase: \( \xi \ne \xi_{1}\lambda^{-1} \).}  We can write \( \xi = \xi_{1}x\xi_{2} \) and
  \( \lambda = \xi_{2}^{-1} \tilde{x} \lambda_{1} \), where \( x \cdot \tilde{x} \ne e \), \( ||\xi_{2}|| = 0 \) or
  \( ||\xi_{2}|| = 1 \), \( ||\lambda_{1}|| = 1 \) or \( ||\lambda_{1}||=0 \), and \( ||\xi_{1}|| \ge 4 \).  Therefore
  \[ \xi \lambda^{r} = \symform{\lpart{\xi_{1}} u \lambda_{1} \xi_{2}^{-1} \tilde{x} \lambda_{1} \cdots
    \rpart{\xi_{2}^{-1} \tilde{x} \lambda_{1} \xi_{2}^{-1} \tilde{x} \lambda_{1}}}.  \]
  \( ||\xi_{1}|| \ge 4 \) implies that \( \xi \) starts with \( \lambda^{-1}\lambda^{-1} \), which is impossible.  This
  finishes the first case.

  \bigskip
  
  \emph{Case: \( \xi \) is of type \( (i,j) \) and \( \lambda \) is of type \( (i,i) \)}.  In this case there are no
  cancellations between \( \xi \) and \( \lambda \), the length of \( \lambda \) is odd, and if
  \( \lambda = \tilde{x}_{1} \lambda_{1} \tilde{x}_{2} \), then the product
  \[ \lambda^{r} = \tilde{x}_{1} \lambda_{1} u \lambda_{1}u \cdots u \lambda_{1}\tilde{x}_{2}, \quad u = \tilde{x}_{2}
  \cdot \tilde{x}_{1} \in \Gamma_{i} \setminus \{e\}, \]
  is reduced and \( u \ne e \), because \( \lambda \) is cyclically reduced by assumption.
  
  \medskip

  \emph{Subcase: \( ||\lambda|| \ge ||\xi|| + 3 \).}  Let \( \lambda = \tilde{x}_{1}\lambda_{1}\tilde{x}_{2} \) and
  \[ \xi \lambda^{r} = \symform{\lpart{\xi} \tilde{x}_{1} \lambda_{1} u \lambda_{1} u \cdots u \rpart{\lambda_{1}
      \tilde{x}_{2}}}.  \]
  For \( \xi \lambda^{r} \) to be in \( S^{c}_{\Gamma} \) we must have \( \xi = \tilde{x}_{2}^{-1} \xi_{1} \) and
  \( \lambda_{1} = \lambda_{2} \xi_{1}^{-1} \),
  \[ \xi \lambda^{r} = \symform{\tilde{x}_{2}^{-1} \xi_{1} \lpart{\tilde{x}_{1} \lambda_{2}} \xi_{1}^{-1} u \lambda_{2}
    \xi_{1}^{-1} u \cdots \rpart{u \lambda_{2}} \xi_{1}^{-1} \tilde{x}_{2}}.  \]
  Thus \( \lambda_{2} = \lambda_{3} \tilde{x}_{1}^{-1} \), \( \lambda_{3} = \lambda_{3}^{-1} \) and \( u = \tilde{x}_{1}
  \).  But \( u = \tilde{x}_{2} \cdot \tilde{x}_{1} \), hence \( \tilde{x}_{2} = e \), which is impossible.

  \medskip

  \emph{Subcase: \( ||\lambda|| = ||\xi|| + 1 \).}  As in the previous subcase for
  \( \lambda = \tilde{x}_{1} \lambda_{1} \tilde{x}_{2} \), \( u = \tilde{x}_{2} \cdot \tilde{x}_{1} \ne e \) and
  \( \xi = \tilde{x}_{2}^{-1} \xi_{1} \), \( ||\xi_{1}|| = ||\lambda_{1}|| \), we have
  \[ \xi \lambda^{r} = \symform{\tilde{x}_{2}^{-1} \lpart{\xi_{1} \tilde{x}_{1}} \lambda_{1} u \cdots \rpart{u
      \lambda_{1}} \tilde{x}_{2}}.  \] Hence \( \xi_{1} = \lambda_{1}^{-1} \), \( \tilde{x}_{1} = u^{-1} \) and so
  \[ \xi \lambda^{r} = \symform{\tilde{x}_{2}^{-1} \xi_{1} \tilde{x}_{1} \lpart{\xi_{1}^{-1}} \tilde{x}_{1}^{-1} \cdots
    \rpart{\xi_{1}^{-1}} \tilde{x}_{1}^{-1} \xi_{1}^{-1} \tilde{x}_{2}}, \]
  and we arrive at \( \xi_{1} = \xi_{1}^{-1} \), which is impossible since \( \xi \) is unbalanced and
  \( \xi = \tilde{x}^{-1} \xi_{1} \).

  \medskip

  \emph{Subcase: \( 5 \le ||\lambda|| \le ||\xi|| -1 \).}  As before
  \( \lambda = \tilde{x}_{1}\lambda_{1} \tilde{x}_{2} \), \( \xi = \tilde{x}_{2}^{-1} \xi_{1} \) and
  \[ \xi\lambda^{r} = \symform{\tilde{x}_{2}^{-1} \lpart{\xi_{1}} \tilde{x}_{1} \lambda_{1} u \cdots \rpart{u \lambda_{1}}
    \tilde{x}_{2}}.  \]
  From \( ||\lambda|| \le ||\xi|| - 1 \) we deduce \( \xi_{1} = \lambda_{1}^{-1} u^{-1} \xi_{2} \) and therefore
  \[ \xi\lambda^{r} = \symform{\tilde{x}_{2}^{-1} \lambda_{1}^{-1} u^{-1} \lpart{\xi_{2}} \tilde{x}_{1} \lambda_{1} u
    \cdots u \rpart{\lambda_{1}} u \lambda_{1} \tilde{x}_{2}}.  \]
  Since \( ||\lambda_{1}|| \ge 3 \), if \( ||\xi_{2}|| \ge 2 \) or if \( ||\xi_{2}|| = 1 \) and
  \( \xi_{2} \in \Gamma_{1} \), then the word \( \xi \) is not unbalanced; so \( \xi_{2} \in \Gamma_{2} \) and
  \( \lambda_{1} = \lambda_{2} \tilde{x}_{1}^{-1}\xi_{2}^{-1} \) and
  \[ \xi\lambda^{r} = \symform{\tilde{x}_{2}^{-1} \xi_{2} \tilde{x}_{1} \lambda_{2}^{-1} u^{-1} \xi_{2} \tilde{x}_{1}
    \lpart{\lambda_{2} \tilde{x}_{1}^{-1}} \xi_{2}^{-1} u \cdots \rpart{u \lambda_{2}} \tilde{x}_{1}^{-1} \xi_{2}^{-1} u
    \lambda_{2} \tilde{x}_{1}^{-1} \xi_{2}^{-1} \tilde{x}_{2}}.  \]
  This implies \( \lambda_{2} = \lambda_{2}^{-1} \) and \( u = \tilde{x}_{1} \).  Contradiction.

  \medskip

  \emph{Subcase: \( ||\lambda|| = 3 \).}  Let \( \lambda = \tilde{x}_{1} \tilde{y} \tilde{x}_{2}\), then
  \[ \xi \lambda^{r} = \symform{\lpart{\xi} \tilde{x}_{1} \tilde{y} u \tilde{y} u \cdots \rpart{u \tilde{y} u \tilde{y}
    \tilde{x}_{2}}}.  \]
  Therefore the reduced form of \( \xi \) start with
  \( \tilde{x}_{2}^{-1} \tilde{y}^{-1} u^{-1} \tilde{y}^{-1} u^{-1} \), and therefore cannot be unbalanced.

  \bigskip

  \emph{Case: \( \xi \) is of type \( (i,j) \) and \( \lambda \) is of type \( (j,j) \).}  In this case there has to be
  a ``full cancellation'': either the reduced form of \( \lambda^{r} \) starts with \( \xi^{-1} \) or the reduced form
  of \( \xi \) ends in the reduced form of \( \lambda^{-r} \).  For if this were not the case, then
  \( \xi = \xi_{1} x \xi_{2} \) and \( \lambda^{r} = \xi_{2}^{-1}\tilde{x} \lambda_{1} \) with
  \( x \cdot \tilde{x} \ne e \), hence \( \xi \lambda^{r} = \xi_{1} u \lambda_{1} \) is reduced with
  \( u = x \cdot \tilde{x} \), and therefore \( \xi \lambda^{r} \) has type \( (i,j) \), and in particular it is not in
  \( S^{c}_{\Gamma} \).

  \medskip

  \emph{Subcase: \( ||\lambda|| \ge ||\xi|| + 3 \).}  Let \( \xi = \xi_{1}x \),
  \( \lambda = x^{-1}\xi_{1}^{-1}\lambda_{1} \tilde{x}_{2} \), where \( x, \tilde{x}_{2} \ne e \) and
  \[ \xi \lambda^{r} = \symform{\lpart{\lambda_{1} u} \xi_{1}^{-1}\lambda_{1} u \cdots \xi_{1}^{-1}\rpart{\lambda_{1}
      \tilde{x}_{2}}}, \quad u = \tilde{x}_{2} \cdot x^{-1} \in \Gamma_{j} \setminus\{e\}.  \]
  One sees that \( \lambda_{1} = \tilde{x}_{2}^{-1} \lambda_{2} \), \( \lambda_{2}^{-1} = \lambda_{2} \),
  \( u = \tilde{x}_{2} \), and, as usually, this implies \( x^{-1} = e \).

  \medskip

  \emph{Subcase: \( ||\lambda|| = ||\xi|| + 1 \).}  In this subcase \( \xi = \xi_{1}x \) and
  \( \lambda = x^{-1}\xi_{1}^{-1}\tilde{x}_{2} \) and
  \[ \xi \lambda^{r} = \symform{\lpart{u\xi_{1}^{-1}} u\xi_{1}^{-1}u \cdots u \rpart{\xi_{1}^{-1}\tilde{x}_{2}}}.  \]
  Hence \( u = \tilde{x}_{2}^{-1} \) and \( \xi_{1}^{-1} = \xi_{1} \), which is a contradiction.

  \medskip

  \emph{Subcase: \( 5 \le ||\lambda|| \le ||\xi|| - 1 \).}  Let
  \( \lambda = \tilde{x}_{1} \tilde{y} \lambda_{1}\tilde{x}_{2} \) and
  \( \xi = \xi_{1} y u^{-1} \lambda_{1}^{-1} \tilde{y}^{-1} \tilde{x}_{1}^{-1} \), where
  \( u = \tilde{x}_{2} \cdot \tilde{x}_{1} \).  Note that \( ||\lambda_{1}|| \ge 2 \).
  As we argued at the beginning of the case, there has to a full cancellation, and we must have
  \( y \cdot \tilde{y} = e \), and therefore also \( y \in \Gamma_{2} \) (otherwise \( \xi \) is not unbalanced).  So
  \[  \xi \lambda^{r} = \xi_{1} \cdot \lambda_{1} u y^{-1} \lambda_{1} u \cdots u y^{-1} \lambda_{1} \tilde{x}_{2},  \]
  where the product is not necessarily reduced as the dot indicates.  If \( ||\xi_{1}|| > 0 \), then the last letter of
  \( \xi_{1} \) is in \( \Gamma_{1} \) (since \( y \in \Gamma_{2} \)); this letter has to cancel with the first letter
  of \( \lambda_{1} \), but this implies that \( \xi \) is not unbalanced.  Therefore \( ||\xi_{1}|| = 0 \) and
  \[ \xi \lambda^{r} = \symform{\lpart{\lambda_{1} u} \tilde{y} \lambda_{1} u \cdots u \tilde{y} \rpart{\lambda_{1}
      \tilde{x}_{2}}}, \]
  which implies \( \lambda_{1} = \tilde{x}_{2}^{-1} \lambda_{2} \), \( \lambda_{2} = \lambda_{2}^{-1} \),
  \( u = \tilde{x}_{2} \) and again \( \tilde{x}_{1} = e \), which is impossible.

  \medskip

  \emph{Subcase: \( ||\lambda|| = 3 \).}  Let \( \lambda = \tilde{x}_{1} \tilde{y} \tilde{x}_{2} \), and
  \( \xi = \xi_{1} z_{0} y u^{-1} \tilde{y}^{-1} \tilde{x}_{1}^{-1} \).  Since \( \xi \) is unbalanced, either
  \( y \cdot \tilde{y} \ne e \) or \( z_{0} \cdot u \ne e \), and therefore \( \xi \lambda^{r} \) has type \( (i,j) \)
  and is not in \( S^{c}_{\Gamma} \).

  \bigskip

  \emph{Case: both \( \xi \) and \( \lambda \) are of type \( (i,j) \)}.  In this case the product
  \( \xi \lambda^{r} \) is reduced and has type \( (i,j) \), and so is not in \( S^{c}_{\Gamma} \).
\end{proof}

\begin{lemma}
  \label{lem:balanced-times-anyelement-is-not-in-S}
  Let \( \xi \in \Gamma_{1}\freepr\Gamma_{2} \) be an unbalanced element of asymmetric type, \( ||\xi|| \ge 6 \).  If
  \( \gamma \not \in S^{c}_{\Gamma} \) and \( r \ge 4 \), then \( \xi \gamma^{r} \not \in S^{c}_{\Gamma} \).
\end{lemma}

\begin{proof}
  The proof is again by contradiction: assume that we do have \( \xi \), \( \gamma \) and \( r \) such that the
  assumptions hold and \( \xi \gamma^{r} \in S^{c}_{\Gamma} \).  Let \( \gamma = \alpha \lambda \alpha^{-1} \) with
  \( \lambda \) being cyclically reduced, \( ||\lambda|| \ge 2 \), and \( \xi\gamma^{r} = \beta z \beta^{-1} \) for some
  \( z \in X_{\Gamma} \).  First of all we note that it cannot happen that \( \xi = \xi_{1}x \xi_{2} \) and
  \( \alpha = \xi_{2}^{-1}\tilde{x}\alpha_{1} \) for some \( \xi_{1}, \xi_{2}, \alpha_{1} \in \Gamma \) and some
  non-trivial multipliable \( x, \tilde{x} \in X_{\Gamma} \) with \( x \cdot \tilde{x} \ne e \).  Indeed, if this were
  the case, then the product
  \[  \xi \gamma^{r} = \xi_{1}u\alpha_{1}\lambda^{r} \alpha_{1}^{-1}\tilde{x}^{-1}\xi_{2}, \quad u = x \cdot \tilde{x},  \]
  is reduced and therefore \( \xi \gamma^{r} \) has asymmetric type, and cannot be an element in \( S^{c}_{\Gamma} \).
  So either the product \( \xi \alpha \) is reduced, or one of \( \xi \), \( \alpha \) completely cancels the
  other one.  More formally, we have the following cases.

  \emph{Case: \( \xi \gamma^{r} \) is reduced.}  In this case the product \( \xi \alpha \lambda^{r} \alpha^{-1} \)
  is reduced (here \( \lambda^{r} \) is understood as a single element, not as a product).
  
  \emph{Subcase: \( ||\xi|| \ge ||\alpha|| \).}  For \( \xi \alpha \lambda^{r} \alpha^{-1} \) to be of the form
  \( \beta z \beta^{-1}\) we must have \( \xi = \alpha \xi_{1} \) and therefore
  \( \xi \gamma^{r} = \alpha \xi_{1} \alpha \lambda^{r} \alpha ^{-1} = \alpha( \xi_{1} \alpha \lambda^{r}) \alpha^{-1}
  \) .  Note that \( \xi = \alpha \xi_{1} \) is asymmetric and unbalanced if and only if so is \( \xi_{1} \alpha \).  We
  therefore apply Lemma \ref{lem:balanced-times-power-of-reduced-not-is-in-S} to \( \xi_{1} \alpha \) and \( \lambda \).

  \emph{Subcase: \( ||\xi|| < ||\alpha|| \).}  Let \( K\ge 0 \) be the largest such that
  \( \alpha = \xi^{K}\alpha_{1} \).  Then
  \( \xi \gamma^{r} = \xi \xi^{K}\alpha_{1} \lambda^{r} \alpha_{1}^{-1} \xi^{-K} = \xi^{K}(\xi \alpha_{1} \lambda ^{r}
  \alpha_{1}^{-1})\xi^{-K} \).  The set \( S^{c}_{\Gamma} \) is conjugation invariant, therefore we need to consider
  \( \xi \alpha_{1} \lambda^{r} \alpha_{1}^{-1} \) only.  We claim that \( ||\alpha_{1}|| < ||\xi|| \).  If this were
  not the case and \( ||\alpha_{1}|| \ge ||\xi|| \), then for \( \xi \alpha_{1} \lambda^{r} \alpha_{1}^{-1} \) to be of
  the form \( \beta z \beta^{-1} \) we must have \( \alpha_{1} = \xi \alpha_{2} \), contradicting the maximality of
  \( K \).  So \( ||\alpha_{1}|| < ||\xi|| \) and we are under the conditions of the previous subcase with
  \( \alpha_{1} \) in the place of \( \alpha \).

  \emph{Case: Full reduction.}

  \emph{Subcase: \( ||\xi|| \ge ||\alpha|| \) and therefore \( \xi = \xi_{1}\alpha^{-1} \).}  In this subcase
  \( \xi \gamma^{r} = \xi_{1} \cdot \lambda^{r} \alpha^{-1} \), where the product is not necessarily reduced, and
  therefore \( \alpha^{-1} \xi \gamma^{r} \cdot \alpha = \alpha^{-1}\xi_{1} \cdot \lambda^{r} \), where
  \( \alpha^{-1}\xi_{1} \) is asymmetric and unbalanced, because so is \( \xi = \xi_{1} \alpha^{-1} \). It remains to
  apply Lemma \ref{lem:balanced-times-power-of-reduced-not-is-in-S}.
  
  \emph{Subcase: \( ||\xi|| < ||\alpha|| \) and therefore \( \alpha = \xi^{-1}\alpha_{1} \)}.  Let \( K \ge 1 \) be
  maximal such that \( \alpha = \xi^{-K}\alpha_{2} \), then
  \( \xi \gamma^{r} = \xi^{-K+1}\alpha_{2}\lambda^{r}\alpha_{2}^{-1}\xi^{K} \).  It is enough to show that
  \( \xi \cdot \alpha_{2} \lambda^{r} \alpha_{2}^{-1} \not \in S^{c}_{\Gamma} \).  By maximality of \( K \) we conclude that
  \( ||\alpha_{2}|| \le ||\xi|| \) and therefore can apply the previous subcase with \( \alpha_{2} \) in the role of
  \( \alpha \).
\end{proof}

\section{Automatic continuity for homomorphisms into free products}
\label{sec:autom-cont-homom}

In this section we shall finally prove Theorem \ref{thm:continuity-into-free-products}.  As we have mentioned earlier,
the core of the argument is to show that there are no \emph{surjective} homomorphisms \( \phi : G \to \Gamma \) with
\emph{dense kernels}.  Depending on the structure of \( \Gamma \), we shall use somewhat different modifications of
Dudley's argument.  One of the special cases will be \( \Gamma = \mathbb{Z}_{2} \freepr \mathbb{Z}_{2} \).  This case
is, in fact, already known.  More precisely, Rosendal \cite[Theorem 3.2]{MR2849256} in a somewhat different terminology
has proved

\begin{theorem}[Rosendal]
  \label{thm:continuity-into-dihedral-group}
  For a completely metrizable topological group \( G \) and any homomorphism
  \( \phi : G \to \mathbb{Z}_{2}\freepr\mathbb{Z}_{2} \) one of the two possibilities hold: either \( \phi \) is
  continuous, or \( \phi(G) \) is contained in one of the factors.
\end{theorem}

In particular this implies that there does not exist a surjective homomorphism from a completely metrizable group onto
\( \mathbb{Z}_{2} \freepr \mathbb{Z}_{2} \) with a dense kernel.  We shall need this corollary and for the purpose of
completeness we give a proof of it using the setting of this paper.  This argument is, in fact, a reformulation of
Rosendal's original proof.

In the rest of the section \( G \) denotes a completely metrizable topological group.

\begin{lemma}
  \label{lem:dense-kernel-surjective-dihedral}
  Let \( \phi : G \to \mathbb{Z}_{2} \freepr \mathbb{Z}_{2} \) be a surjective homomorphism.  The kernel of \( \phi \)
  cannot be dense in \( G \).
\end{lemma}

\begin{proof}
  Suppose the kernel of \( \phi \) were dense.  By surjectivity of \( \phi \) there is some \( g \in G \) such that
  \( \phi(g) = xy \), where \( x \) and \( y \) are the generators of the two factors in
  \( \mathbb{Z}_{2} \freepr \mathbb{Z}_{2} \).  Let \( K \) be the kernel of \( \phi \), and so \( gK \) is dense in
  \( G \).  We construct a sequence \( (g_{m})_{m=1}^{\infty} \) of elements in \( G \) such that \( \phi(g_{m}) = xy
  \), and the limit
  \[ h_{m} = \lim_{n \to \infty} g_{m}(g_{m+1}(\cdots (g_{n-1}(g_{n})^{r_{n-1}})^{r_{n-2}} \cdots)^{r_{m+1}})^{r_{m}} \]
  exists, where \( r_{m} = 2^{m} \).  We claim that \( ||\phi(h_{m})^{r_{m-1}}|| \ge r_{m-1} \) for any \( m \ge 2 \)
  and \( \phi(h_{m}) \in S_{\mathbb{Z}_{2}\freepr\mathbb{Z}_{2}} \).  By continuity \( h_{m} = g_{m}h_{m+1}^{r_{m}} \)
  and therefore \( \phi(h_{m})=xy\phi(h_{m+1})^{r_{m}} \).  If
  \( \phi(h_{m+1}) \in S^{c}_{\mathbb{Z}_{2} \freepr \mathbb{Z}_{2}}\), then \( \phi(h_{m+1}) \) has order \( 2 \) and
  therefore \( \phi(h_{m+1})^{r_{m}} = e \).  In this case \( ||\phi(h_{m})^{r_{m-1}}|| = (xy)^{r_{m-1}} = 2r_{m-1} \).
  If \( \phi(h_{m+1}) \in S_{\mathbb{Z}_{2} \freepr \mathbb{Z}_{2}} \), then \( \phi(h_{m+1}) = (xy)^{q} \) for some
  non-zero integer \( q \), and therefore \( ||\phi(h_{m})^{r_{m-1}}|| = ||(xy)^{r_{m-1}(qr_{m} +1)}|| \ge r_{m-1} \).
  This proves the claim and therefore \( ||\phi(h_{1})|| \ge m \) for any \( m \) by the same argument as the end of
  Theorem \ref{thm:free-products-are-not-completely-metrizable}.  This contradiction proves the lemma.
\end{proof}

\begin{lemma}
  \label{lem:dense-kernel-surjective-prime-order}
  Let \( \phi : G \to \Gamma \), where \( \Gamma = \Gamma_{1}\freepr\Gamma_{2} \), be a surjective homomorphism and
  assume that \( \Gamma_{1} \freepr \Gamma_{2} \ne \mathbb{Z}_{2} \freepr \mathbb{Z}_{2} \) and that
  \( \Gamma_{1}\freepr\Gamma_{2} \) does not have elements of order \( p \) for a prime \( p \).  The kernel
  of \( \phi \) cannot be dense in \( G \).
\end{lemma}

\begin{proof}
  Suppose the kernel of \( \phi \) were dense.  By Lemma \ref{lem:S-squared-is-not-Gamma} and surjectivity of \( \phi \)
  there is \( g \in G \) such that \( \phi(g) \not \in S^{c}_{\Gamma} \freepr S^{c}_{\Gamma} \).  Using the density of
  the kernel, and hence the density of any coset, we construct a sequence \( (g_{m})_{m=1}^{\infty} \) in \( G \) such
  that \( \phi(g_{m}) = \phi(g) \not \in S^{c}_{\Gamma} \cdot S^{c}_{\Gamma} \) and the limit
  \[ h_{m} = \lim_{n \to \infty} g_{m}(g_{m+1}(\cdots (g_{n-1}(g_{n})^{r_{n-1}})^{r_{n-2}} \cdots)^{r_{m+1}})^{r_{m}} \]
  exists, where \( r_{m} = p^{m + m||\phi(g)||} \).  We claim that \( ||\phi(h_{m})^{r_{m-1}}|| \ge r_{m-1} \) for all
  \( m \ge 2 \).  Indeed, if \( \phi(h_{m}) \in S_{\Gamma} \), then this is true by the definition of the set
  \( S_{\Gamma} \).  If \( \phi(h_{m}) \in S^{c}_{\Gamma} \), then
  \( \phi(g_{m}) \not \in S^{c}_{\Gamma} \cdot S^{c}_{\Gamma} \) implies \( \phi(h_{m+1}) \in S_{\Gamma} \) and
  therefore \( ||\phi(h_{m})|| = ||\phi(g_{m})\phi(h_{m+1})^{r_{m}}|| \ge r_{m} - ||\phi(g)|| \ge r_{m-1}\).  Finally,
  \( r_{m-1} \) is a power of \( p \) and therefore \( ||\phi(h_{m})^{r_{m-1}}|| \ge ||\phi(h_{m})|| \)  by item
  \eqref{item:when-power-is-not-shorter} of Lemma \ref{lem:description-of-S-and-F}, which proves
  the claim.  This inequalities imply \( ||\phi(h_{1})|| \ge m \) as before and the lemma follows.
\end{proof}

\begin{lemma}
  \label{lem:dense-kernel-surjectvie-unbalanced}
  Let \( \phi : G \to \Gamma \), where \( \Gamma = \Gamma_{1}\freepr\Gamma_{2} \), be a surjective homomorphism and
  assume there is an unbalanced \( \xi \in \Gamma_{1}\freepr\Gamma_{2} \) of asymmetric type and length
  \( ||\xi|| \ge 6 \).  The kernel of \( \phi \) cannot be dense in \( G \).
\end{lemma}

\begin{proof}
  Suppose the kernel of \( \phi \) were dense.  Let \( \xi \in \Gamma_{1}\freepr\Gamma_{2} \) be an unbalanced element of
  asymmetric type and of length at least \( 6 \).  By surjectivity of \( \phi \) we can pick \( g \in G \) such that
  \( \phi(g) = \xi \).  Using the density of the kernel construct a sequence
  \( (g_{m})_{m=1}^{\infty} \) such that \( \phi(g_{m}) = \xi \) and the limit
  \[ h_{m} = \lim_{n \to \infty} g_{m}(g_{m+1}(\cdots (g_{n-1}(g_{n})^{r_{n-1}})^{r_{n-2}} \cdots)^{r_{m+1}})^{r_{m}} \]
  exists, where \( r_{m} = m + m||\xi|| \).  Since \( \xi \) is of asymmetric type, unbalanced and \( ||\xi|| \ge 6 \),
  Lemma \ref{lem:asymmetric-unbalanced-is-not-in-Scomp-squared} implies
  \( \xi \not \in S^{c}_{\Gamma} \cdot S^{c}_{\Gamma}\).  Since \( \phi(h_{m}) = \xi \phi(h_{m+1})^{r_{m}} \), Lemma
  \ref{lem:balanced-times-anyelement-is-not-in-S} then implies \( \phi(h_{m}) \in S_{\Gamma} \) for all \( m \ge 2 \),
  and we again arrive at an absurd inequality \( ||\phi(h_{1})|| \ge m \) for all \( m \).
\end{proof}

Before the final spurt we would like to recall classical Kurosh's subgroup theorem:

\begin{theorem}[Kurosh]
  \label{thm:Kurosh-subgroup-theorem}
  Any subgroup \( \Gamma' \) of a free product \( \Gamma = \freeprod \) can be written as
  \( F(Z) \freepr \Bigl(\fp{b}{B}{\Gamma'}\Bigr) \), where \( Z \subseteq \Gamma \) generates a free group \( F(Z) \)
  and for each \( b \in B \) there are \( a \in A \) and \( \gamma_{b} \in \Gamma \) such that
  \( \Gamma_{b}' = \gamma_{b} \Lambda_{b} \gamma_{b}^{-1} \) and \( \Lambda_{b} < \Gamma_{a} \).
\end{theorem}

\begin{theorem}
  \label{thm:homomorphisms-into-free-products-are-continuous-or-contain-in-a-factor}
  If \( \phi : G \to \Gamma \) is a homomorphism from a completely metrizable group \( \Gamma \) into a discrete group
  \( \Gamma = \freeprod \), then either \( \phi \) is continuous or the image of \( \phi \) is contained in one of the
  factors of \( \Gamma \).
\end{theorem}

\begin{proof}
  Let \( K \) denote the kernel of \( \phi \). The closure of the kernel \( \closure{K} \) is also a completely
  metrizable group and we may consider the restriction of \( \phi \) onto \( \closure{K} \).  If \( \closure{K} = K \),
  then Theorem \ref{cor:closed-kernel-implies-continuity} implies \( \phi \) is continuous and we assume that
  \( \closure{K} \ne K \).  In this case the group \( \Gamma' = \phi(\closure{K}) \) is non-trivial and we can apply
  Kurosh's subgroup theorem to write \( \Gamma' = F(Z)\freepr(\fp{b}{B}{\Gamma'}) \).  The homomorphism
  \( \phi : \closure{K} \to \Gamma' \) is surjective and has a dense kernel.  We claim that \( \phi(\closure{K}) \) is
  contained in a factor of \( \Gamma' \); in other words we claim that \( \Gamma' \) has only one factor.  Suppose this
  is false.  Pick two factors of \( \Gamma' \) and let \( \pi \) be the canonical projection from \( \Gamma' \) onto the
  free product of these two factors.  The map \( \pi \circ \phi : \closure{K} \to \pi(\Gamma') \) is surjective, has a
  dense kernel and its image is a free product of two nontrivial groups.  By Lemma
  \ref{lem:dense-kernel-surjective-dihedral} the group \( \pi(\Gamma') \) is not the infinite dihedral group and by
  Lemma \ref{lem:dense-kernel-surjectvie-unbalanced} there are no unbalanced asymmetric words of length at least \( 6 \)
  in \( \pi(\Gamma') \).  But then there can be at most \( 4 \) elements of order bigger than \( 2 \) in each of the
  factors of \( \pi(\Gamma') \), for if there were \( 5 \) such elements in say \( \Gamma_{1} \), then one could select
  three of them \( x_{1}, x_{2}, x_{3} \) such that \( x_{i} \ne x_{j}^{\pm 1}\) for \( i \ne j \) and so
  \( x_{1}yx_{2}yx_{3}y \) is an asymmetric unbalanced element of length \( 6 \) for any non-trivial
  \( y \in \Gamma_{2} \).  Hence there must be a prime \( p \) such that neither \( \Gamma_{1} \) nor \( \Gamma_{2} \)
  have elements of order \( p \) and therefore Lemma \ref{lem:dense-kernel-surjective-prime-order} applies and gives a
  contradiction.

  Thus \( \phi(\closure{K}) \) is contained in a factor of \( \Gamma' \), which by Dudley's Theorem
  \ref{thm:Dudley-theorem} cannot be the free factor, therefore \( \phi(\closure{K}) \) is contained in a factor of
  \( \Gamma \), i.e., \( \phi(\closure{K}) \subseteq \gamma \Gamma_{a_{0}} \gamma^{-1} \) for some
  \( \gamma \in \Gamma \) and some \( a_{0} \in A \).  We can write \( \Gamma = \fp{a}{A}{\widetilde{\Gamma}} \), where
  \( \widetilde{\Gamma}_{a} = \gamma \Gamma_{a} \gamma^{-1} \), so there is no loss in generality to assume that
  \( \gamma = e \).

  Assume first that \( \phi : G \to \Gamma \) is surjective.  We claim that in this case \( \Gamma = \Gamma_{a_{0}} \).
  Indeed, if \( \card{A} \ge 2 \) and there is \( b \in A \) such that \( b \ne a_{0} \), then pick \( g \in G \) such
  that \( \phi(g) \in \Gamma_{b}\setminus\{e\} \) and choose an \( h \in \closure{K}\setminus K \).  Since
  \( \closure{K} \) is a normal subgroup of \( G \), we must have \( \phi(ghg^{-1}) \in \Gamma_{a_{0}} \), but
  \( \phi(g)\phi(h)\phi(g)^{-1} = x_{b}x_{a_{0}}x_{b}^{-1} \) for non-trivial \( x_{a_{0}} \in \Gamma_{a_{0}} \) and
  \( x_{b} \in \Gamma_{b} \).  This is a contradiction.

  So the theorem is proved under the assumption that \( \phi \) is surjective.  If \( \phi \) is not
  surjective, let \( \Gamma' \) be its image.  Using Kurosh's subgroup theorem and applying what we have proved for the
  now surjective \( \phi : G \to \Gamma' \) we derive that either \( \phi \) is continuous (in which case it is also
  continuous viewed as a map \( \phi : G \to \Gamma \)) or the image of \( \phi \) is contained in a factor of
  \( \Gamma' \).  But any factor of \( \Gamma' \) is contained up to a conjugation in a factor of \( \Gamma \), except
  possibly for the free group factor \( F(Z) \).  But if the image of \( \phi \) is a free group, then \( \phi \) is
  continuous by Dudley's Theorem \ref{thm:Dudley-theorem}.
\end{proof}

\bibliographystyle{alpha}
\bibliography{automatic-continuity}

\end{document}